\documentclass[
a4paper,
11pt,
DIV=14,
toc=bibliography, 
headsepline, 
parskip=half-,
abstract=true,
]{scrartcl}

\usepackage[centertags]{amsmath}
\usepackage[utf8]{inputenc}
\usepackage{enumerate,amsfonts,amssymb,amsthm,amsopn,cite,array,comment}

\usepackage[usenames]{color}
\definecolor{citegreen}{rgb}{0,0.6,0}
\definecolor{refred}{rgb}{0.8,0,0}
\usepackage[colorlinks, citecolor=citegreen, linkcolor=refred, urlcolor=black]{hyperref}

\title{Ends of (singular) Ricci shrinkers}
\author{Alessandro Bertellotti and Reto Buzano}
\date{}

\newtheorem{theorem}{Theorem}[section]
\newtheorem{lemma}[theorem]{Lemma}
\newtheorem{cor}[theorem]{Corollary}
\newtheorem{prop}[theorem]{Proposition}

\theoremstyle{definition}
\newtheorem{rem}[theorem]{Remark}

\newtheorem{defn}[theorem]{Definition}
\numberwithin{equation}{section}
\newtheorem{claim}[theorem]{Claim}

\newcommand{\n}{\mathcal{N}}
\newcommand{\N}{\mathbb{N}}
\newcommand{\R}{\mathbb{R}}
\newcommand{\e}{\mathsf{E}}
\newcommand{\X}{\mathcal{X}}
\newcommand{\s}{\mathcal{S}}
\newcommand{\reg}{\mathcal{R}}
\newcommand{\m}{\mathfrak{M}}
\newcommand{\h}{\mathcal{H}}
\newcommand{\E}{\mathcal{E}}
\newcommand{\f}{\mathsf{F}}
\renewcommand{\r}{\mathsf{r}}
\newcommand{\kn}{\mathbin{\bigcirc\mspace{-15mu}\wedge\mspace{3mu}}}
\newcommand{\Crit}{\mathrm{Crit}}
\DeclareMathOperator{\Ric}{Ric}
\DeclareMathOperator{\Rm}{Rm}
\DeclareMathOperator{\Riem}{Riem}
\DeclareMathOperator{\Vol}{Vol}
\DeclareMathOperator{\diam}{diam}

\newcommand\printaddress{{
\setlength{\parindent}{15pt}
\footnotesize~
\par
{\scshape Alessandro Bertellotti}
\newline 
SISSA, 
Via Bonomea 265, 
34136 Trieste, Italy
\newline
\emph{E-mail address:} 
\texttt{abertell@sissa.it}
\par
{\scshape Reto Buzano}
\newline 
Universit\`a di Torino, 
Dipartimento di Matematica,
Via Carlo Alberto 10, 
10123 Torino, Italy 
\newline
\emph{E-mail address:} 
\texttt{reto.buzano@unito.it}
\par
}}

\begin{document}
\maketitle
\begin{abstract}
We estimate the number of ends of smooth and singular Ricci shrinkers focussing first on general ends and later on asymptotically conical ones. In particular, we obtain a variety of applications to sequences of Ricci shrinkers converging in a weak pointed sense to a possibly singular limit Ricci shrinker, for instance no new conical end can form in the limit. 
\end{abstract}

\section{Introduction and main results}\label{introduction}

A \emph{smooth} gradient shrinking Ricci soliton, \emph{Ricci shrinker} for short, is a complete, connected Riemannian manifold $(M,g)$ together with a smooth function $f\in C^{\infty}(M)$, called potential, satisfying 
\begin{equation}\label{eq.shrinker}
\Ric+\nabla^{2}f=\frac{1}{2} g.
\end{equation}
Besides being natural generalisations of positive Einstein manifolds, Ricci shrinkers play a prominent role in the study of the Ricci flow. They are self-similar solutions to the Ricci flow equation
\begin{equation*}
\frac{\partial }{\partial t}g_{t}=-2\Ric(g_{t})
\end{equation*}
and model finite time singularities. Indeed, if the Ricci flow is of Type I, any blow up sequence at a singular point will smoothly sub-converge to a smooth non-flat Ricci shrinker \cite{Na10,EMT11,MM15}. Local singularities are modelled by non-compact Ricci shrinkers, and their geometry at infinity can provide some useful information about the original manifold. For example, if one could flow past the singularity in a weak sense, the number of components that the manifold is expected to separate into (at least in a suitable local sense) should be the same as the number of ends of the singularity model. Hence asking how many ends a Ricci shrinker can have is a natural question. We refer to Section \ref{sectionends} for precise definitions of the set of ends $\mathcal{E}$ (see Definitions \ref{def.end1} and \ref{def.end2}) and explanations on how to count their number properly. 

Some results about estimating the number of ends of a Ricci shrinker can be found in \cite{msw} and \cite{mw22}. In particular, in \cite{mw22} Munteanu-Wang proved that if an $n$-dimensional Ricci shrinker has scalar curvature $S\leq \frac{n}{3}$ then it has only one end. The same conclusion was recently obtained by Li-Wang \cite{kahler} if the Ricci shrinker is a K\"ahler surface. Without these assumptions, a Ricci shrinker can have several ends, the cylinder $\R \times \mathbb{S}^{n-1}$ being the most basic example of a shrinker with more than one end. While it is known that gradient \emph{steady} Ricci solitons are either connected at infinity or else are isometric to $\R \times N$, with $N$ compact and Ricci-flat \cite{mw11,mw14}, a similar result for Ricci shrinkers is not known. Indeed, it is not clear to us whether one should even expect such a result -- for example, an analogous result for shrinkers of the mean curvature flow does not hold (see for example \cite{IW,Ket} for an analytical example, or \cite{BNS} for a different numerical one).

In this article we will also consider \emph{singular} Ricci shrinkers defined as follows.

\begin{defn}\label{singularRS}  
A \emph{singular space} of dimension $n$ is a tuple $(\X,d,\s,g)$ such that
\begin{enumerate}
\item $(\X,d)$ is a locally compact, complete and connected metric length space with a singular-regular decomposition $\X=\s \cup \reg$, where the regular set $\reg=\X \setminus \s$ is open and dense and has the structure of a differentiable $n$-manifold;
\item $(\reg,g)$ is an incomplete Riemannian manifold with induced metric $d_g$ coinciding with $d\vert_{\reg}$ and $(\X,d)$ is its metric completion;
\item for every compact subset $K\subseteq \X$ and every $R>0$ there are positive constants $0<\kappa_{1}(K,R) < \kappa_{2}(K,R)$ such that for all $x\in K$ and $0<r<R$
\begin{equation*}
\kappa_{1}r^{n}\leq \Vol(B(x,r)\cap \reg)\leq \kappa_{2}r^{n}.
\end{equation*}
\end{enumerate}
We say that $(\X,d,\s,g)$ has \emph{mild singularities} if in addition
\begin{enumerate}
\item[4.] for any $x\in \reg$ there is a closed subset $Q_x$ of measure zero such that for any $y \in \reg \setminus Q_x$ there is a minimising geodesic between $x$ and $y$ whose image lies in $\reg$.
\end{enumerate}
Finally, a \emph{singular Ricci shrinker} is a singular space $(\X,d,\s,g)$ with mild singularities such that $(\reg,g)$ is an incomplete gradient Ricci shrinker (i.e.~there exists a smooth function $f \in C^{\infty}(\reg)$ satisfying 
\eqref{eq.shrinker} on the regular part) with nonnegative scalar curvature $S\geq 0$.
\end{defn}

In the following we will often denote a singular space, or a singular Ricci shrinker, simply with $\X$, implying the other data. 

The main reason for studying singular Ricci shrinkers comes from two compactness results: first, it is proved by Li-Li-Wang \cite{LLW} and Huang-Li-Wang \cite{HLW} that under a uniform lower bound on the entropy (defined below), a sequence of Ricci shrinkers always sub-converges to a singular Ricci shrinker in the pointed singular Cheeger-Gromov sense; see Definition \ref{def.singularCG} and Theorem \ref{LLWcompactness}. This generalises earlier work of the second author and Haslhofer in \cite{hm,hm4}. Second, Bamler \cite{Ba20_compactness,Ba20_structure} has more recently generalised the result on blow-up limits at Ricci flow singularities mentioned above to Ricci flows without the Type I assumption, working with a weak notion of convergence and with singular Ricci shrinkers. We remark that in Bamler's results the mildness condition from Point 4 of Definition \ref{singularRS} has only been announced so far (see Remark 2.17 in \cite{Ba20_structure}), but was previously obtained under the additional assumption of a scalar curvature bound \cite{bamler,Ba17} and it therefore seems natural to us to include it.
We note here that the concept of singular space is a slight modification of the one of \emph{conifold}, first appearing in \cite[Definition~1.2]{chen} in the context of
K\"ahler-Ricci flow. 

Such compactness results might give good \emph{local} control of the sequence. For example, in \cite{by} the second author and Yudowitz have refined the compactness results from \cite{hm,hm4} to obtain a local diffeomorphism finiteness result for Ricci shrinkers. Since the convergence is of pointed type, and therefore only on \emph{compact} subsets, this theory does not give any information about the possible geometry at infinity, and in particular about the number of ends. For example, the limit of a converging sequence of Ricci shrinkers could have more or fewer ends than the elements in the sequence. Therefore we are interested in estimating the number of ends both of a single (smooth or singular) Ricci shrinkers as well as of the limit of a converging sequence. 

Before stating our main results, let us recall a few facts for smooth Ricci shrinkers. More details as well as some corresponding results for singular Ricci shrinkers are given in Section \ref{section-preliminaries}. Every smooth Ricci shrinker comes with a natural basepoint $p\in M$ where the potential $f$ attains its minimum, and the growth of $f$ is controlled by a quarter of the distance from $p$ squared \cite{hm}. Moreover, the volume of geodesic balls centred at $p$ is at most Euclidean. These facts allow to normalise $f$, by adding a constant, such that
\begin{equation}\label{normalisation}
\int_{M}e^{-f}d\nu_{g}=(4\pi)^{\frac{n}{2}}.
\end{equation}
We always assume this normalisation. Under \eqref{normalisation}, the shrinker has a well-defined entropy
\begin{equation*}
\mu_{g}:=\int_{M}(|\nabla f|^{2}+S+f-n)(4\pi)^{-\frac{n}{2}}e^{-f}d\nu_{g}.
\end{equation*}
The entropy was introduced by Perelman \cite{perelman} and for a general Ricci flow it is non decreasing, therefore assuming a lower bound for it on Ricci shrinkers is quite natural. As a second consequence of the growth estimates for $f$, we can use
\begin{equation*}
\varrho=2\sqrt{f-\mu_g}
\end{equation*}
to measure distances. Our first results are the following.

\begin{theorem}\label{thm.endbound} 
Given $n \in \mathbb{N}$, $\underline{\mu}>-\infty$, and $r_0>0$, there exists $\Lambda=\Lambda(n,\underline{\mu},r_{0})$ with the following property. If $M$ is a smooth $n$-dimensional Ricci shrinker with entropy bounded below $\mu\geq \underline{\mu}$, and such that $|\nabla f|>0$ on $\{\varrho > r_0\}$, then $M$ has at most $\Lambda$ ends.
\end{theorem}

\begin{theorem}\label{thm.limitendbound}
Let $(M_{i},g_{i},f_{i},p_{i})_{i\in \N}$ be a sequence of pointed, smooth $n$-dimensional Ricci shrinkers converging to a singular $n$-dimensional Ricci shrinker $(\X,d,\s,g,f,p)$ in the pointed singular Cheeger-Gromov sense. Suppose there exists $r_{0}>0$ such that $|\nabla f_{i}|>0$ on $\{\varrho_i >r_0\}$ for all $i\in \N$. Then the following hold:
\begin{enumerate} 
\item If there exists a compact subset $K\subseteq \X$ such that $\s\cap (\X\setminus K)$ is discrete then 
\begin{equation*}
\#\mathcal{E}(\X)\leq \liminf_{i\to \infty}\#\mathcal{E}(M_{i}).
\end{equation*}
\item If $|\nabla f|>0$ on $\{\varrho> r_0\}$ and $\s$ is contained in a compact subset then 
\begin{equation*}
\#\mathcal{E}(\X) = \lim_{i\to \infty}\#\mathcal{E}(M_{i}).
\end{equation*} 
\end{enumerate}
\end{theorem}

We refer to Sections \ref{section-preliminaries} and \ref{sectionends} for the precise definitions and properties of the concepts and quantities appearing in the above results. We have stated Theorem \ref{thm.endbound} for smooth Ricci shrinkers here, but the result can be generalised to singular Ricci shrinkers in which case the constant $\Lambda$ will additionally also depend on $\kappa_{1}$ and $\kappa_{2}$ from Definition \ref{singularRS}, see Remark \ref{rem.11singular} for more details.

We point out that working with Ricci shrinkers without critical points at large distance is a reasonable assumption that follows for example from bounded scalar curvature or more generally if the scalar curvature satisfies the growth estimate
\begin{equation*}
S(x)\leq \alpha d(x,p)^{2}+C
\end{equation*}
for every $x\in M$, where $\alpha \in (0,1/4)$ and $C>0$, see \cite{hm}. 

The preceding theorems are combined with the compactness theory of \cite{hm,hm4,LLW,HLW} in Section \ref{applications-convergence} to derive informations about some moduli spaces of Ricci shrinkers. We will prove that the moduli spaces of Ricci shrinkers with precisely $N$ ends is compact (or pre-compact), meaning no ends can form, disappear, or merge, under a variety of different assumptions, see Corollaries \ref{cor1}--\ref{cor4}.

From a geometric viewpoint, \emph{asymptotically conical ends} are of particular interest. We recall the precise definition (see also Section \ref{section-conicalends} for more details).

\begin{defn}\label{asymptconical} 
Let $\X$ be a singular space. An end $E$ of $\X$ relative to some compact subset $K$ is \emph{$C^{k}$-asymptotically conical} if $E\subseteq \reg$ and there exist a Riemannian cone $(\mathcal{C}_{\Sigma},g_c)$, a constant $a>0$ and a diffeomorphism 
$\varphi\colon \mathcal{C}_{\Sigma}(a,\infty)\to E$ such that
\begin{equation}\label{conicalconditionmetric}
\sup_{x\in \mathcal{C}_{\Sigma}(a,\infty)}|\varphi^{\ast}g-g_{c}|(x)<\infty
\end{equation}
and
\begin{equation}\label{conicalcondition}
\lim_{b\to \infty}\sup_{x\in \mathcal{C}_{\Sigma}(b,\infty)}\r(x)^{\ell}\vert\nabla^{\ell}(\varphi^{\ast}g-g_{c})\vert(x)=0
\end{equation}
for all $0\leq \ell\leq k$. 
\end{defn}

In \cite{kw}, Kotschwar-Wang proved that two Ricci shrinkers having ends asymptotic to the same cone are isometric along the ends in question, while they are globally isometric when simply connected. In \cite{mw15,mw19}, Munteanu-Wang relate the behaviour at infinity of curvature with the presence of conical ends. Regarding (smooth or singular) asymptotically conical Ricci shrinkers as singularity models of the Ricci flow, there is interesting progress in flowing through the singularities in a weak sense using for example Ricci flow with conical/orbifold singularities or using expanding solitons asymptotic to the same cone, see for example \cite{simon,GS18,AK22} and references therein.

Our main results regarding conical ends are the following.

\begin{theorem}\label{thm.conicalbound}
Given $n \in \mathbb{N}$ and a finite collection $\h$ of connected compact Riemannian manifolds of dimension $n-1$, there exists $\Lambda=\Lambda(n,\h)$ with the following property. Let $\X$ be a smooth (or limit singular) $n$-dimensional Ricci shrinker and suppose that every asymptotically conical end of $\X$ has asymptotic cone $\mathcal{C}_{\Sigma}$ for some $\Sigma\in \h$. Then $\X$ has at most $\Lambda$ asymptotically conical ends.
\end{theorem}

\begin{theorem}\label{thm.limitconicalbound} 
Let $(M_{i},g_{i},f_{i},p_{i})_{i \in \N}$ be a sequence of pointed, smooth $n$-dimensional Ricci shrinkers converging to a singular 
$n$-dimensional Ricci shrinker $\X$ in the pointed singular Cheeger-Gromov sense. Then
\begin{equation*}\#\mathcal{E}_{c}(\X)\leq \liminf_{i\to \infty}\#\mathcal{E}_{c}(M_{i}).\end{equation*}
\end{theorem}

Note that Theorem \ref{thm.limitconicalbound} does not contain any assumption about critical points of $f$, so Theorem \ref{thm.limitendbound} does not apply and theoretically the sequence might contain ``growing'' regions that form a new end in the limit, but the result says that such a new end can never be asymptotically conical. 

Note also that the theorem allows for fewer asymptotically conical ends in the limit, even in cases where the total number of ends remains constant. This could for example happen if asymptotically conical ends along a sequence ``collapse'' to a cylindrical end in the limit. Numerical evidence of such a scenario happening has been found for the mean curvature flow in \cite{BNS}.

Again, we will prove several applications of Theorem \ref{thm.limitconicalbound} to moduli spaces, see Section \ref{conicalends-applications-section}.

Let us quickly discuss the strategy of the proof of Theorem \ref{thm.limitconicalbound}. Suppose that the limit of a converging sequence of shrinkers has an asymptotically conical end and take a compact set intersecting it in a large region. Since the convergence is on compact subsets, we can find corresponding regions inside the late shrinkers of the sequence that are nearly conical in the following sense.

\begin{defn}\label{nearannulus}
Let $\X$ be a singular space and $\Omega \subseteq \X$ an open subset. We say that $\Omega\subseteq \X$ is \emph{$C^{k}$-near to the conical annulus} $\mathcal{C}_{\Sigma}(a,b)$, where $0<a<b\leq \infty$, if $\Omega\subseteq \reg$ and there exists a diffeomorphism 
$\varphi\colon \mathcal{C}_{\Sigma}(a,b)\to \Omega$ such that
\begin{equation} \label{closenesscondition}
\sup_{\mathcal{C}_{\Sigma}(a,b)}a^{\ell}\vert \nabla^{\ell}(\varphi^{\ast}g-g_{c})\vert \leq \varepsilon 
\end{equation}
for all $0\leq \ell\leq k$, where $\varepsilon \in (0,1/2)$.
\end{defn} 

We then show the following result, which could be of independent interest. A smooth Ricci shrinker containing a nearly conical region far away from the basepoint is asymptotically conical along an end containing this region. More precisely, we prove the following theorem.

\begin{theorem}\label{thm.conicalannulus}
Let $n\in \N$ and let $\Sigma$ be a compact connected Riemannian manifold of dimension $n-1$ with cone 
$\mathcal{C}_{\Sigma}$.
There exists a positive constant $a_{\ast}=a_{\ast}(n,\Sigma)>0$ with the following property.
Let $(M,g,f)$ be a smooth $n$-dimensional Ricci shrinker with basepoint $p\in M$. 
Suppose there exists a connected open subset $\Omega\subseteq M$ such that:
\begin{enumerate}
\item $\Omega$ is $C^{2}$-near to a conical annulus $\mathcal{C}_{\Sigma}(a,b)$ as in Definition \ref{nearannulus};
\item $\Omega$ contains a connected component $V$ of $\{\alpha < \varrho < \beta\} \neq \emptyset$.
\end{enumerate}
If $a_{\ast} \leq a \leq \frac{1}{3}b$ and $a_{\ast} \leq \alpha \leq \frac{1}{3} \beta$ then $M$ is not compact and there exists a unique end $\e\in \mathcal{E}(M)$ such that $V\subseteq \e(\{\varrho \leq \alpha\})$. Moreover $\e$ is asymptotically conical and $\nabla f \neq 0$ on $\e(\{\varrho \leq \alpha\})$.
\end{theorem}

Note that \cite{heatkernel,liwang} contain some results of similar nature.

The article is organised as follows. In Section \ref{section-preliminaries} we collect some preliminary facts about smooth and singular Ricci shrinkers, and recall the notions of convergence we need. Section \ref{sectionends} is devoted to the concept of ends of a space; we give precise definitions and collect some useful properties. In Section \ref{section-estimate-number-ends} we prove Theorems \ref{thm.endbound} and \ref{thm.limitendbound}. Applications of these theorems to moduli spaces of Ricci shrinkers are then obtain in Section \ref{applications-convergence}. 
Section \ref{section-conicalends} concentrates on asymptotically conical ends and their geometric properties as well as nearly conical regions in the sense of Definition \ref{nearannulus}. Section \ref{nearlyconical-section} contains the proof of Theorem \ref{thm.conicalannulus}, while in Section \ref{conicalends-estimate-section} we prove Theorems \ref{thm.conicalbound} and \ref{thm.limitconicalbound}. We then obtain further applications to moduli spaces in Section \ref{conicalends-applications-section}. 
Finally, Appendix \ref{appendixends} contains the proofs of the propositions from Section \ref{sectionends}, while Appendix \ref{appendix-nearbymetrics} is a reference for some useful facts regarding nearby Riemannian metrics used throughout the article. 

\textbf{Acknowledgements.} AB has been fully supported by a Phd studentship from SISSA.
RB has been partially supported by the Italian PRIN project ``Differential-geometric aspects of manifolds via global analysis'' (No.~20225J97H5). The authors also thank the referee for valuable comments.

\section{Preliminaries on smooth and singular Ricci shrinkers}\label{section-preliminaries}

Let $M=(M,g,f)$ be a smooth Ricci shrinker. It is well known that the quantity
\begin{equation}\label{auxilliaryconstant}
c(g)=S+|\nabla f|^{2}-f
\end{equation}
is constant on $M$, see e.g. \cite{chow}. Moreover, the scalar curvature $S$ is always non negative and only vanishes at some point if $M$ is isometric to $\R^{n}$ \cite{rps}. As a consequence, $|\nabla f|^2 \leq f+c(g)$ and hence
\begin{equation}\label{eq.rhoLipschitz}
|\nabla \sqrt{f+c(g)}| \leq \tfrac{1}{2}.
\end{equation}
As stated in the introduction, the growth of the potential $f$ is controlled by the distance from the basepoint $p$ of the shrinker. 
More precisely, one has
\begin{equation}\label{fgrowth}
\frac{1}{4}(d(x,p)-5n)_{+}^{2}\leq f(x)+c(g) \leq \frac{1}{4}(d(x,p)+\sqrt{2n})^{2},
\end{equation}
where  $a_{+}=\max\{a,0\}$. From this, one sees that two different minimum points of $f$ lie at most at distance $5n+\sqrt{2n}$ from each other. Setting $\varrho=2\sqrt{f+c(g)}$, the inequality \eqref{fgrowth} becomes
\begin{equation}\label{rhogrowth}
(d(x,p)-5n)_{+}\leq \varrho(x)\leq d(x,p)+\sqrt{2n}.
\end{equation}
In particular, by \eqref{eq.rhoLipschitz} and \eqref{rhogrowth}, $\varrho$ is $1$-Lipschitz and proper. If $M$ is not flat then $\varrho >0$ and smooth everywhere, while it coincides with $d(x,p)$ if $M=\R^{n}$.

Metric balls with center at $p\in M$ satisfy the volume growth condition
\begin{equation}\label{volumegrowth}
\Vol(B(p,r))\leq C_{0}r^{n},\quad \forall r>0,
\end{equation}
where $C_{0}=C_{0}(n)$ is a positive constant. Estimate \eqref{fgrowth} together with the Euclidean growth \eqref{volumegrowth} imply that the weighted volume of $M$ is finite, $\int_{M}e^{-f}d\nu_{g}<\infty$, and therefore we can normalise $f$ as in \eqref{normalisation}. Under this normalisation the constant $c(g)$ is equal to minus the Perelman entropy of $(M,g)$:
\begin{equation}
-c(g)=\mu(g)=\int_{M}(|\nabla f|^{2}+S+f-n)(4\pi)^{-\frac{n}{2}}e^{-f}d\nu_{g}.
\end{equation} 
Finally, under the additional assumption of an entropy lower bound $\mu(g)\geq \underline{\mu}$, it is possible to prove a volume non-collapsing result. In this case, for every $r>0$ and $0<\delta\leq 1$, if $B(x,\delta)\subseteq B(p,r)$, we have
\begin{equation}\label{noncollapsing}
\Vol(B(x,\delta))\geq \kappa(r)\delta^{n},
\end{equation}
where $\kappa(r)=\kappa(r,n,\underline{\mu})$ is a positive function. Proofs of \eqref{fgrowth}--\eqref{noncollapsing} can be found in \cite{hm}.

Every smooth Ricci shrinker has an associated Ricci flow, which can be constructed as follows. Since $M$ is complete, the vector field $\nabla f$ is complete \cite{zhang} and therefore generates a global flow $\Phi\colon M\times \R \to M$ with associated one-parameter group of diffeomorphisms 
$(\varphi_{t})_{t\in \R}$ satisfying
\begin{equation*}
\left\{
\begin{aligned}
\tfrac{d}{dt}\varphi_{t}(x) &=\nabla f(\varphi_{t}(x)),\\
\varphi_{0}(x) &=x,
\end{aligned}
\right.
\end{equation*}
for every $x\in M$ and $t\in \R$.
Similarly the time-dependent vector field $\frac{1}{1-t}\nabla f$, where $t\in (-\infty,1)$, is complete in the sense that it generates a family of diffeomorphisms 
$(\psi_{t})_{t\in (-\infty,1)}$ satisfying
\begin{equation*}
\left\{
\begin{aligned}
\tfrac{d}{dt}\psi_{t}(x) &=\tfrac{\nabla f(\psi_{t}(x))}{1-t},\\
\psi_{0}(x) &=x,
\end{aligned}
\right.
\end{equation*}
for every $x\in M$ and $t\in (-\infty,1)$. Indeed one can easily pass from one flow to the other, using the relations
\begin{equation}\label{cambioparametro}
\psi_{t}(x)=\varphi_{-\log(1-t)}(x), \quad \forall x\in M,\ \forall t\in (-\infty,1),
\end{equation}
and
\begin{equation}\label{cambioparametro2}
\varphi_{t}(x)=\psi_{1-e^{-t}}(x), \quad \forall x\in M,\ \forall t\in \R.
\end{equation}
We will use these flows throughout the article. Defining the metrics $g_{t}=(1-t)\psi_{t}^{\ast}g$ one can prove that 
$(g_{t})_{t\in (-\infty,1)}$ is an ancient Ricci flow on $M$ and since the metrics defined in this way are self-similar, the quantities depending on them can be expressed in term of $g$ only. Indeed, the $(1,3)$-curvature tensor satisfies $\Rm(x,t)=\psi_{t}^{\ast}\Rm(x)$ and we therefore obtain:
\begin{align*}
|\Rm(x,t)| &=(1-t)^{-1}|\Rm(\psi_{t}(x))|,\\
S(x,t) &=(1-t)^{-1}S(\psi_{t}(x)),\\
B_{t}(x,r) &=\psi_{t}^{-1}\left(B(\psi_{t}(x),r/\sqrt{1-t})\right),
\end{align*}
where the quantities on the left hand side refer to the evolving metric $g_{t}$, while the ones on the right hand side refer to the fixed metric $g$. Finally, the potential $f$ satisfies the evolution equations
\begin{equation}
\frac{d}{dt}(f\circ \varphi_{t})=|\nabla f \circ \varphi_{t}|^{2}
\end{equation}
for every $t\in \R$ and
\begin{equation}\label{evolution2}
\frac{d}{dt}(f\circ \psi_{t})=\frac{|\nabla f\circ \psi_{t}|^{2}}{1-t}
\end{equation}
for every $t\in (-\infty,1)$.\\

In the remainder of this section, we concentrate our attention to singular Ricci shrinkers as in Definition \ref{singularRS}. Some of the above properties for smooth shrinkers generalise to the singular case, with some precautions.

\begin{prop} Let $\X$ be a singular Ricci shrinker. The following hold:
\begin{enumerate}
\item The quantity $c(g)=|\nabla f|^{2}-f+S$ is constant on $\reg$.
\item The function $f$ (defined on $\reg$) can be extended continuously to the entire $\X$, 
attains its minimum at some point $p \in \X$ and satisfies the growth condition:
\begin{equation}\label{fgrowth-singular}
\frac{1}{4}\big(d(x,p)-2\sqrt{f(p)+c(g)}-4n\big)^{2}_{+}\leq f(x)+c(g)\leq \frac{1}{4}\big(d(x,p)+2\sqrt{f(p)+c(g)}\big)^2
\end{equation}
for every $x\in \X$. In particular, $f$ can be normalised as in \eqref{normalisation}.
\item $\varrho=2\sqrt{f+c(g)}$ is $1$-Lipschitz and proper and satisfies
\begin{equation}\label{rhogrowth-singular}
\left(d(x,p)-\varrho(p)-4n\right)_{+} \leq \varrho(x)\leq d(x,p)+\varrho(p)
\end{equation} 
for every $x\in \X$.
\end{enumerate}
\end{prop}

\begin{proof}
The first point is obvious, being $(\reg,g)$ a connected Ricci shrinker. In order to prove the other two points, note that from the first point and the assumption $S\geq 0$, we conclude \eqref{eq.rhoLipschitz} or equivalently $|\nabla \varrho|\leq 1$ on $\reg$. Hence $\varrho$ is $1$-Lipschitz with respect to $d_g \equiv d\vert_{\reg}$ and therefore can be extended continuously to the completion $\X$. In particular we can also extend $f$.

Now we follow the proof of Lemma $2.1$ in \cite{hm}. First of all, since $\varrho$ is $1$-Lipschitz we have
\begin{equation}\label{ineq1}
\sqrt{f(x)+c(g)}\leq \frac{1}{2}\big(d(x,y)+ 2\sqrt{f(y)+c(g)}\big)
\end{equation}
for every $x,y\in \X$. Now fix a point $x \in \reg$ and let $Q_x$ be the set of measure zero as in the mildness condition of Definition \ref{singularRS}. For all $y\in \reg\setminus Q_x$, we can consider a minimal geodesic $\gamma\colon [0,d]\to \reg$ connecting $x$ to $y$, where $d=d(x,y)$. Exactly as in \cite{hm}, using the second variation formula for the energy of $\gamma$ and the shrinker equation, one finds
\begin{equation}\label{ineq2}
\frac{1}{2}d(x,y)+\frac{4}{3}-2n \leq \sqrt{f(x)+c(g)}+\sqrt{f(y)+c(g)}+1.
\end{equation}
Since $\reg\setminus Q_x$ is dense, we can extend the preceding inequality first to all $y$ and then to all $x$ in $\X$ by (Lipschitz) continuity.

Now let $(x_{k})$ be a minimising sequence for $f$. Thanks to \eqref{ineq2} it follows that $(x_{k})$ is bounded and therefore sub-converges to a point $p$ of $\X$ by completeness, and this point is a minimum of $f$. Evaluating the inequalities \eqref{ineq1} and \eqref{ineq2} at $p$ we get
\begin{equation*}
\frac{1}{2}\big(d(x,p)-2\sqrt{f(p)+c(g)}-4n\big) \leq \sqrt{f(x)+c(g)} \leq \frac{1}{2}\big(d(x,p)+2\sqrt{f(p)+c(g)}\big),
\end{equation*}
and the desired growth conditions \eqref{fgrowth-singular} and \eqref{rhogrowth-singular} follow. By Point 3 of Definition \ref{singularRS} and the growth condition just obtained, $f$ can be normalised in order to satisfy \eqref{normalisation}. Finally $\varrho$ is proper, thanks to \eqref{rhogrowth-singular} and since $\X$ is compactly bounded, being a complete length space.
\end{proof}

Several of our results require a particular notion of convergence. We start recalling the definition of $\varepsilon$-approximation.

\begin{defn}\label{approssimazione} 
A map $f\colon (X,d_{X},p)\to (Y,d_{Y},q)$ between pointed metric spaces is an $\varepsilon$-pointed Gromov-Hausdorff approximation ($\varepsilon$-pGHA) if it is quasi-isometric and quasi-onto in the following sense:
\begin{enumerate}
\item $|d_{Y}(f(x),f(y))-d_{X}(x,y)|<\varepsilon$ for every $x,y \in B_{X}(p,1/\varepsilon)$;
\item for every $y\in B_{Y}(q,1/\varepsilon)$ there exists $x\in B_{X}(p,1/\varepsilon)$ with $d_{Y}(f(x),y)<\varepsilon$.
\end{enumerate}
\end{defn}

We are ready to give the precise definition we need. 

\begin{defn}\label{def.singularCG} 
A sequence $(M_{i},g_{i},f_{i},p_{i})_{i \in \N}$ of pointed, $n$-dimensional smooth Ricci shrinkers converges to a pointed, $n$-dimensional singular Ricci shrinker $(\X,d,\s,g,f,p)$ in the pointed singular Cheeger-Gromov sense if
\begin{enumerate}
\item there exists an exhaustion $(U_{i})_{i \in \N}$ of the regular part $\reg$ and a family of embeddings $\phi_{i}\colon U_{i}\to M_{i}$
such that $(\phi_{i}^{\ast}g_{i},\phi_{i}^{\ast}f_{i})$ converges to $(g,f)$ in $C^{\infty}_{\text{loc}}(\reg)$;
\item the maps $\phi_{i}\colon U_{i}\to M_{i}$ can be extended to $\varepsilon_{i}$-pGHA between $(\X,d,p)$ and $(M_{i},d_{g_{i}},p_{i})$ 
with $\varepsilon_{i}\to 0$. 
\end{enumerate}
\end{defn}

Singular shrinkers arising as limits of smooth Ricci shrinkers as in the previous definition will be called \emph{limit singular Ricci shrinkers}. Note that their basepoint $p$ and potential $f$ come naturally from the sequence. In this case we can strengthen the previous proposition.

\begin{prop}\label{prop.limitsingularRS}
Let $\X$ be a limit singular Ricci shrinker. Then $f$ attains its minimum at $p$ and the inequalities
\eqref{fgrowth-singular}--\eqref{rhogrowth-singular} can be improved to \eqref{fgrowth}--\eqref{rhogrowth}
for every $x\in \X$. Moreover the distance between two minimum points of $f$ is bounded by $5n+\sqrt{2n}$. If in addition the sequence converging to $\X$ satisfies the entropy lower bound $\mu_{i}\geq \underline{\mu}$ then:
\begin{enumerate}
\item $f$ is normalised, that is $\int_{\reg}e^{-f}d\nu_{g}=(4\pi)^{\frac{n}{2}}$;
\item $\Vol(B(p,r)\cap \reg)\leq C_{0}r^{n}$ for every $r>0$, where $C_{0}$ is as in \eqref{volumegrowth};
\item $\Vol(B(x,\delta)\cap \reg)\geq \kappa(r) \delta^{n}$ for every $B(x,\delta)\subseteq B(p,r)$, where $\delta \in (0,1]$ and
$\kappa$ is as in \eqref{noncollapsing}.
\end{enumerate}
\end{prop}

\begin{proof}
Essentially, all results follow by smooth convergence on $\reg$, but for completeness we give a detailed proof here. Let $\X=(\X,d,\s,g,f,p)$ be the limit of a sequence $(M_{i},g_{i},f_{i},p_{i})_{i \in \N}$. Now for every $x\in \X$, Point 2 in Definition \ref{approssimazione} implies that $d(\widehat{\phi}_{i}(x),p_{i})\to d(x,p)$ as $i\to \infty$, where $\widehat{\phi}_{i}$ is the extension of $\phi_{i}$ to an $\varepsilon_{i}$-pGHA. For a fixed $x\in \reg$ we have $x\in U_{i}$ for $i$ large enough and
\begin{equation*}
\frac{1}{4}(d(\phi_{i}(x),p_{i})-5n)_{+}^{2}\leq f_{i}(\phi_{i}(x))+c(g_{i}) \leq \frac{1}{4}(d(\phi_{i}(x),p_{i})+\sqrt{2n})^{2}.
\end{equation*}
Passing to the limit in the preceding inequality, noting that the auxiliary constants $c(g_{i})$ converge to $c(g)$ thanks to the smooth convergence on $\reg$, we get \eqref{fgrowth} or equivalently \eqref{rhogrowth} on $\reg$. Since $\reg$ is dense in $\X$, these estimates extend by uniform continuity to the whole of $\X$.

To see that $p\in \X$ is actually a minimum point for $f$, let $(x_{k})\subseteq \reg$ be a sequence converging to $p$. If $k\in \N$ is fixed then $x_{k}\in U_{i}$ for $i$ large and since every $\varrho_{i}$ is $1$-Lipschitz we have
\begin{equation*}
\varrho_{i}(\phi_{i}(x_{k})) - d(\phi_{i}(x_{k}),p_{i}) \leq \varrho_{i}(p_{i})\leq \varrho_{i}(\phi_{i}(x_{k})) + d(\phi_{i}(x_{k}),p_{i}).
\end{equation*}
Taking limits, first in $i$ and then in $k$, we get
\begin{equation*}
\lim_{i\to \infty}\varrho_{i}(p_{i})=\varrho(p).
\end{equation*}
This obviously implies $f_{i}(p_{i})\to f(p)$ and it is easy to conclude. Finally the statement about the distance between minimum points is a simple consequence of \eqref{fgrowth}.

Now we work under the additional condition $\mu(g_{i})\geq \underline{\mu}$ for every $i \in \N$.
\begin{enumerate}
\item This follows from the smooth convergence on $\reg$ and decay estimates for $e^{-f}$. See Proposition $8.8$ in \cite{LLW} for a detailed proof.

\item Consider the metric ball $B(p,r)$ and let $K\subseteq B(p,r)\cap \reg$ be a compact subset. 
Then $K\subseteq U_{i}$ for $i$ large and we can consider $\phi_{i}(K)\subseteq M_{i}$. 
It is easy to see that $\phi_{i}(K)\subseteq B(p_{i},r+\varepsilon_{i})$ and therefore
$\Vol_{g_{i}}(\phi_{i}(K))\leq C_{0}(r+\varepsilon_{i})^{n}$, thanks to \eqref{volumegrowth}. Now
\begin{equation*}
\Vol_{g}(K)\leq \vert \Vol_{g}(K)- \Vol_{\phi_{i}^{\ast}g_{i}}(K)\vert + \Vol_{g_{i}}(\phi_{i}(K))
\leq C(K)|g-\phi_{i}^{\ast}g_{i}|_{g}+C_{0}(r+\varepsilon_{i})^{n}.
\end{equation*}
Taking the limit as $i\to \infty$ we get $\Vol_{g}(K)\leq C_{0}r^{n}$. Since the Riemannian measure on $\reg$ is regular, 
taking the supremum over $K$ we obtain $\Vol_{g}(B(p,r)\cap \reg)\leq C_{0}r^{n}$, as desired.

\item Denote by $\n_{\varepsilon}(A)$ and $\overline{\n}_{\varepsilon}(A)$ the open and the closed $\varepsilon$-neighbourhoods of a set $A$, respectively. Consider $B(x,\delta)\subseteq B(p,r)$, where $\delta \in (0,1]$, and fix $\varepsilon>0$ small such that  $B(x,\delta)\setminus \overline{\n}_{\varepsilon/2}(\s)$ is not empty. Since $\overline{B(p,r)}\setminus \n_{\varepsilon/4}(\s)$ is compact and contained in $\reg$ we have $\overline{B(p,r)}\setminus \n_{\varepsilon/4}(\s)\subseteq U_{i}$ for $i$ large. We can also suppose $\varepsilon_{i}<\min\{\varepsilon/10, \delta/5, 1/2r \}$ for such $i$. An argument similar to the one we will see in Claim \ref{cc} below shows that
\begin{equation*}
B(x_{i},\delta-5\varepsilon_{i}) \setminus \overline{\n}_{\varepsilon}\big(\widehat{\phi}_{i}(\s)\big)
\subseteq \phi_{i}\big(B(x,\delta)\setminus \overline{\n}_{\varepsilon/2}(\s)\big),
\end{equation*}
where $x_{i}=\widehat{\phi}_{i}(x)$. Therefore, we have 
\begin{align*}
\Vol_{g}(B(x,\delta)\cap \reg)
&\geq \big(\Vol_{g}\big(B(x,\delta)\setminus \overline{\n}_{\varepsilon/2}(\s)\big)
-\Vol_{\phi_{i}^{\ast}g_{i}}\big(B(x,\delta)\setminus \overline{\n}_{\varepsilon/2}(\s)\big) \big) \\
&\quad+\Vol_{\phi_{i}^{\ast}g_{i}}\big(B(x,\delta)\setminus\overline{\n}_{\varepsilon/2}(\s)\big).
\end{align*}
The first term can be estimated by
\begin{equation*}
\vert \Vol_{g}(B(x,\delta)\setminus \overline{\n}_{\varepsilon/2}(\s))
-\Vol_{\phi_{i}^{\ast}g_{i}}(B(x,\delta)\setminus \overline{\n}_{\varepsilon/2}(\s))\vert
\leq C_{\varepsilon}|\phi_{i}^{\ast}g-g|_{g},\end{equation*}
where $C_{\varepsilon}>0$ is a constant depending only on $B(x,\delta)\setminus \overline{\n}_{\varepsilon/2}(\s)$. Regarding the second term, write
\begin{align*}
\Vol_{\phi_{i}^{\ast}g_{i}}\big(B(x,\delta)\setminus\overline{\n}_{\varepsilon/2}(\s)\big)
&\geq Vol_{g_{i}}\big(B(x_{i},\delta-5\varepsilon_{i})\setminus \overline{\n}_{\varepsilon}\big(\widehat{\phi}_{i}(\s)\big)\big) \\
&=\Vol_{g_{i}}(B(x_{i},\delta-5\varepsilon_{i})) -\Vol_{g_{i}}\big(\overline{\n}_{\varepsilon}\big(\widehat{\phi}_{i}(\s)\big)\cap B(x_{i},\delta-5\varepsilon_{i})\big).
\end{align*}
First of all, observe that
\begin{equation*}
\overline{\n}_{\varepsilon}\big(\widehat{\phi}_{i}(\s)\big)\cap B(x_{i},\delta-5\varepsilon_{i})\subseteq 
\n_{\varepsilon}\big(\widehat{\phi}_{i}\big(\s\cap B(p,r+1)\big)\big).
\end{equation*}
In the proof of Proposition $8.2$ in \cite{LLW}, it is shown that $\n_{\varepsilon}(\s)\cap B(p,r+1)$ can be covered by at most 
$N$ balls of radius $3\varepsilon$, 
\begin{equation*}
\n_{\varepsilon}(\s)\cap B(p,r+1) \subseteq \bigcup_{\alpha=1}^{N}B(y_{\alpha},3\varepsilon),
\end{equation*}
where $N\leq C\varepsilon^{3-n}$ and $C$ is a positive constant depending only on $n$ and $r$. In particular
\begin{equation*}
\s \cap B(p,r+1) \subseteq \bigcup_{\alpha=1}^{N}B(y_{\alpha},3\varepsilon)
\end{equation*}
and it follows that 
\begin{equation*}
\n_{\varepsilon}\big(\widehat{\phi}_{i}(\s \cap B(p,r+1))\big)\subseteq \bigcup_{\alpha=1}^{N}B(y_{\alpha}^{i},5\varepsilon),
\end{equation*}
where $y_{\alpha}^{i}=\widehat{\phi}_{i}(y_{\alpha})$. Therefore
\begin{align*}
\Vol_{g_{i}}\big(\overline{\n}_{\varepsilon}\big(\widehat{\phi}_{i}(\s)\big)\cap B(x_{i},\delta-5\varepsilon_{i})\big)
&\leq \Vol \big(\n_{\varepsilon}\big(\widehat{\phi}_{i}(\s \cap B(p,r+1))\big)\big)
\leq \sum_{\alpha=1}^{N}\Vol\left(B(y_{\alpha}^{i}, 5\varepsilon)\right)\\
&\leq N 5^{n}\varepsilon^{n} e^{10(r+11n)\varepsilon}\omega_{n}\leq C\varepsilon^{3},
\end{align*}
where we used Lemma $2.3$ from \cite{LLW} and $C>0$ depends only on $n$ and $r$. 
At this point, putting everything together, we get
\begin{equation*}
\Vol_{g}(B(x,\delta)\cap \reg)\geq 
\kappa(r)(\delta-5\varepsilon_{i})^{n}-C\varepsilon^{3}-C_{\varepsilon}|\phi_{i}^{\ast}g-g|_{g},
\end{equation*}
since $B(x_{i},\delta-5\varepsilon_{i})\subseteq B(p_{i},r)$. Taking the limit as $i\to \infty$, with $\varepsilon$ fixed, we obtain 
\begin{equation*}
\Vol_{g}(B(x,\delta)\cap \reg)\geq \kappa(r)\delta^{n}-C\varepsilon^{3},
\end{equation*}
and sending $\varepsilon$ to $0$ the claim follows. \qedhere
\end{enumerate}
\end{proof}

\section{Definitions and basic properties of ends of Ricci shrinkers}\label{sectionends}

We first recall some definitions and facts about ends of a general topological space. Despite being standard results (often implicitly used in articles studying noncompact spaces), we give very precise statements as they seem to be rather difficult to find in the literature. To keep the presentation compact, we moved the proofs of all propositions of this section to Appendix \ref{appendixends}.

Let $X$ be a connected, locally path connected, locally compact, second countable Hausdorff topological space and denote by 
$\mathcal{K}_{X}$ the collection of its compact subsets. 

\begin{defn}\label{def.end1} 
An \emph{end of $X$} is a function $\e \colon \mathcal{K}_{X}\to \mathcal{P}(X)$ that assigns to every $K\in \mathcal{K}_{X}$ a connected component $\e(K)$ of $X\setminus K$ and that satisfies $\e(H)\subseteq \e(K)$ if $K\subseteq H$. 
The set of ends of $X$ is denoted by $\mathcal{E}(X)$.
\end{defn}

In research articles on geometric analysis, one can also find the following definition (see e.g. \cite{li}).

\begin{defn}\label{def.end2} 
Let $K\in \mathcal{K}_{X}$. An \emph{end of $X$ relative to $K$} is a non relatively compact connected component
of $X\setminus K$. We denote by $\mathcal{E}(K,X)$ the set of ends relative to $K$.
\end{defn}

Since we are going to use both of the preceding definitions, it is useful to list their main properties and relations.

\begin{prop}\label{prop.ends-general} 
The following hold.
\begin{enumerate}
\item If $\e\in \mathcal{E}(X)$ then $\e(K)\in \mathcal{E}(K,X)$ for every $K\in \mathcal{K}_{X}$.
\item If $X$ is not compact then for every $K\in \mathcal{K}_{X}$ there exists at least one end relative to $K$. In particular $X$ is compact if and only if $\mathcal{E}(X)$ is empty.
\item Given $K\subseteq H$ compact subsets, the natural map $\mathcal{E}(H,X) \to \mathcal{E}(K,X)$ induced by the inclusion is onto.
\item Given an exhaustion $(K_{i})_{i\in \N}$ of $X$ by compact subsets, the map
\begin{align*}
\mathcal{E}(X) &\to \{(U_{i})_{i \in \N}\in \Pi_{i\in \N}\pi_{0}(X\setminus K_{i})\colon U_{i}\supseteq U_{i+1},\ \forall i \in \N\}\\
\e&\mapsto (\e(K_{i}))_{i\in \N}
\end{align*}
is a bijection. 
\item Let $E$ be an end relative to $K\subseteq X$. Then there exists $\e\in \mathcal{E}(X)$ such that $\e(K)=E$. In particular, if for every $H\supseteq K$ compact there exists a unique end relative to $H$ contained in $E$ then $\e$ is unique.
\end{enumerate}
\end{prop}

There are now at least two ways to count the ends of a space. The first one consists in considering $\#\mathcal{E}(X)$ while
the second one uses Definition \ref{def.end2}, as we now explain. Given $K\in \mathcal{K}_{X}$ the set $\mathcal{E}(X,K)$ is at most countable and we can denote by $n(K)$ its cardinality, obtaining a non-decreasing function $n(\cdot)$. Indeed, we have the following.

\begin{prop}\label{numerofinitofini} 
The function $K\mapsto n(K)$ is non-decreasing with respect to inclusion and attains its supremum when bounded. If the supremum is realised by some $K$ then $n(H)=n(K)$ for every $H\supseteq K$ compact. Finally we always have
\begin{equation*}
\#\mathcal{E}(X)=\sup_{K\in \mathcal{K}_{X}}n(K).
\end{equation*}
\end{prop}

Now we consider pointed singular Ricci shrinkers $\X=(\X,\s,g,f,p)$. In order to study their structure at infinity we will work with the super-level sets of the function $\varrho\colon \X\to \R$. We start with some notations.

\begin{defn}\label{def.D}
Given an interval $I\subseteq [0,\infty)$, define
\begin{align*}
D(I)=\{x\in \X\colon \varrho(x)\in I\}.
\end{align*}
If $I=\{a\}$, we also write
\begin{equation*}
D(a)=D(\{a\})=\{x\in \X\colon \varrho(x)=a\}
\end{equation*}
for the corresponding level set. 
\end{defn}

Thanks to \eqref{rhogrowth-singular}, if $I$ is compact then $D(I)$ is compact too (possibly empty), being $\varrho$ proper. If $D[a,\infty)\subseteq \reg$ and $a$ is a regular value with $D(a)\neq \emptyset$ then $D[a,\infty)$ is a submanifold with boundary of 
$\reg$, with interior $D(a,\infty)$ and boundary $D(a)$. Similar statements hold for other types of intervals. If $a\in [\min_{\X}\varrho,\sup_{\X}\varrho)$ then $D(a)$ is not empty and $\varrho(D(I))=I$ for every interval 
$I\subseteq [\min_{\X}\varrho,\sup_{\X}\varrho)$.

Note that, if $\X$ is a limit singular Ricci shrinkers, we have $0\leq \varrho(p)\leq \sqrt{2n}$ and every sub-level $D[0,a]$ contains at least $p$ if $a>\sqrt{2n}$.

The following result is an analogue of a well known fact from Morse theory. 

\begin{prop}\label{prop.ends-shrinker} 
Let $\X$ be a singular Ricci shrinker and $I\subseteq [0,\infty)$ an interval. Consider a connected component $W$ of $D(I)$ with $W\subseteq \reg$ and not containing any critical point of $f$. Then $\varrho(W)=I$ and the following hold.
\begin{enumerate}
\item $W$ is a $n$-dimensional submanifold with boundary of $\reg$. If $I$ is open then $\partial W=\emptyset$
otherwise $\partial W=D(\partial I \cap I)\cap W \neq \emptyset$, and it is the union of the connected components of $D(\partial I\cap I)$ 
that are contained in $W$. The interior $W^{\circ}$ is given by $D(I^{\circ}) \cap W$ and coincides with the unique connected component of $D(I^{\circ})$ contained in $W$.
\item Take $\xi \in I$. The vector field $\nabla f/|\nabla f|^{2}$ induces a diffeomorphism
$W\simeq W(\xi)\times I(\xi)$, where $W(\xi)=D(\xi)\cap W$ and $I(\xi)=\{t\in \R\colon t=(s^{2}-\xi^{2})/4, s\in I\}$ is an interval. 
In particular $W(\xi)$ is connected, compact and is the unique connected component of the level set $D(\xi)$ that is contained in $W$.
An analogous statement holds for the vector field $\nabla f$.
\item For every sub-interval $J\subseteq I$ the set $W(J)=D(J)\cap W$ is connected and is the unique connected component of $D(J)$ that is contained in $W$. In particular, it satisfies all the properties seen for $W$.
\item If $\sup I=\infty$ then $W^{\circ}$ is an end of $\X$ relative to $D[0,a]$, where $a=\inf I$. 
In particular, $\X$ is not compact and $W^{\circ}$ determines a unique end $\e \in \mathcal{E}(\X)$ such that $\e(D[0,a])=W^{\circ}$.
\end{enumerate}
\end{prop}

\section{Estimating the number of general ends}\label{section-estimate-number-ends}

In this section we estimate the number of ends of Ricci shrinkers and in particular prove Theorems \ref{thm.endbound} and \ref{thm.limitendbound}. First of all, we prove the following preliminary result for singular shrinkers.

\begin{prop}\label{endspecialcase} Let $\X$ be a singular Ricci shrinker and assume that $\s\subseteq D[0,r_{0}]$ and $|\nabla f|>0$ on $D(r_{0},\infty) \neq \emptyset$. Let $a\geq r_{0}$.
\begin{enumerate}
\item Every connected component of $D(a,\infty)$ is an end of $\X$ relative to $D[0,a]$ and therefore we can identify $\mathcal{E}(D[0,a],\X)$ 
with $\pi_{0}(D(a,\infty))$.
\item There is a bijection $\mathcal{E}(\X)\to \pi_{0}(D(a,\infty))$.
\item $\#\mathcal{E}(\X)$ is finite.
\end{enumerate} 
\end{prop}

\begin{proof}\
\begin{enumerate}
\item From Proposition \ref{prop.ends-shrinker} it follows that $D(a,\infty)\neq \emptyset$ and each of its connected components is an end of $\X$ relative to $D[0,a]$.

\item Define a map $\mathcal{E}(\X)\to \pi_{0}(D(a,\infty))$ as $\e\mapsto \e(D[0,a])$. This map is well defined thanks to the preceding point and it is onto thanks to Proposition \ref{prop.ends-general}. Let us show that it is injective. Take ends $\e \neq \f$ and suppose by contradiction that $E=\e(D[0,a])=\f(D[0,a])$. Being distinct ends, there exists $b>a$ such that $\widetilde{E}=\e(D[0,b])\neq \widetilde{F}=\f(D[0,b])$. Therefore $\widetilde{E}\cap \widetilde{F}=\emptyset$ but $\widetilde{E},\widetilde{F}\subseteq E$. This says that $E$ contains two distinct connected components of $D(b,\infty)$, in contradiction to Proposition \ref{prop.ends-shrinker}. 

\item Take $\xi>a$ and consider the level set $D(\xi)$. Then $D(\xi)$ is a compact submanifold and has a finite number of connected components. The inclusion $D(\xi)\subseteq D(a,\infty)$ induces a natural map $\pi_{0}(D(\xi))\to \pi_{0}(D(a,\infty))$, which is a bijection as one can easily verify.\qedhere
\end{enumerate} 
\end{proof}

We can now prove Theorems \ref{thm.endbound} based on a counting argument using the volume estimates \eqref{volumegrowth} and \eqref{noncollapsing}. 

\begin{proof}[Proof of Theorem \ref{thm.endbound}.]
Thanks to Proposition \ref{endspecialcase} we know that $M$ has finitely many ends and that the natural map $\mathcal{E}(M)\to \mathcal{E}(D[0,r_{0}],M)$ is a bijection. Let $E_{1},\ldots, E_{k}$ be the distinct ends of $M$ relative to $D[0,r_{0}]$ and, if 
$I\subseteq (r_{0},\infty)$ is an interval, denote by $E_{i}(I)$ the unique connected component of $D(I)$ that is contained in $E_{i}$. 
For every $i=1,\ldots,k$ choose a point $x_{i}\in E_{i}(\xi)$, where $\xi=r_{0}+1$. Since $\varrho$ is $1$-Lipschitz we have $B(x_{i},1)\subseteq E_{i}(\xi-1,\xi+1)$. In particular the balls $B(x_{1},1),\ldots, B(x_{k},1)$ are pairwise disjoint and contained in $B(p,r)$, where $r=\xi+1+5n$.

Now thanks to \eqref{volumegrowth} and \eqref{noncollapsing} the number of disjoint balls of radius $1$ contained in $B(p,r)$ is at most 
$\Lambda=\Lambda(n,\underline{\mu},r)=C_{0}r^{n}/\kappa(1)$ and therefore the number $k$ of ends is bounded by this $\Lambda$. Finally $r=r_{0}+2+5n$ and we have the claimed dependence.
\end{proof}

\begin{rem}\label{rem.11singular} 
The theorem can be generalised to the singular case with minor modifications. Let $\X$ be a singular Ricci shrinker and suppose $\s\subseteq D[0,r_{0}]$. Then $\#\mathcal{E}(\X)\leq \Lambda$, where now $\Lambda=\Lambda(n,r_{0},\kappa_{1}(K,r),\kappa_{2}(K,r))$, with $K=D(r_{0}+1)$, $r=\varrho(p)+r_{0}+1+4n$ and $\kappa_1, \kappa_2$ are the constants from Definition \ref{singularRS}. If $\X$ is a limit singular Ricci shrinker arising from a sequence of Ricci shrinkers with entropy bounded from below, $\mu_{i}\geq \underline{\mu}$, then by Proposition \ref{prop.limitsingularRS} the dependencies in $\Lambda$ are again the same as in the Theorem \ref{thm.endbound}. 
\end{rem}

We are now interested in seeing what happens to the number of ends when we pass to the limit in a converging sequence of shrinkers. If we suppose that the shrinkers have no critical points at large distance and that the limit space has discrete singularities outside of a compact subset, then the number of ends is a lower semicontinuous function, but it becomes continuous if we further assume that the singular locus of the limit is contained in a compact subset.

\begin{proof}[Proof of Theorem \ref{thm.limitendbound}.]
We start with some notations. Let $(U_{i})_{i\in \N}$ and $(\phi_{i}\colon U_{i}\to M_{i})_{i\in \N}$ be as in Definition \ref{def.singularCG} and denote by $\widehat{\phi}_{i}\colon \X\to M_{i}$ the extension of $\phi_{i}$ to  an $\varepsilon_{i}$-pGHA. Furthermore all the objects depending on $M_{i}$ will be indexed by $i$. We also note that, thanks to Proposition \ref{endspecialcase}, if $D_{i}(r_{0},\infty)\neq \emptyset$ then $M_{i}$ has a finite number of ends and for every $a\geq r_{0}$ the map $\mathcal{E}(M_{i})\to \mathcal{E}(D_{i}[0,a],M_{i})$ is a bijection, where we can identify $\mathcal{E}(D_{i}[0,a],M_{i})$ with $\pi_{0}(D_{i}(a,\infty))$. Finally, we remark that $\X$ is a limit singular Ricci shrinker, and therefore we can use the estimates from Proposition \ref{prop.limitsingularRS}.

\begin{enumerate}
\item The idea is to construct an injective map $\lambda_{i}\colon \mathcal{E}(\X)\to \mathcal{E}(M_{i})$ for every $i$ large. First of all we can suppose that $\X$ is not compact, otherwise the claim is trivial. In particular $D(\xi)\neq \emptyset$ for $\xi>\sqrt{2n}\geq \inf_{\X}\varrho$ and similarly $D(a,\infty)\neq \emptyset$ for every $a>0$.

Note that if we take $x\in \reg$ with $\varrho(x)=r_{0}+1$ then $x\in U_{i}$ for $i$ large. Therefore $\phi_{i}(x)\in M_{i}$ and since $\varrho_{i}(\phi_{i}(x))\to \varrho(x)$ we have $\varrho_{i}(\phi_{i}(x))>r_{0}$ for $i$ large. Therefore $D_{i}(r_{0},\infty)\neq \emptyset$ and every late $M_{i}$ is not compact. For the moment suppose $\mathcal{E}(\X)$ is finite. Then from Proposition \ref{numerofinitofini} it follows that there exists $a_{0}\geq r_{0}$ such that the natural map $\mathcal{E}(\X)\to \mathcal{E}(D[0,a_{0}],\X)$ is a bijection. Furthermore, we can take $a_{0}$ satisfying $K\subseteq D[0,a_{0}-1]$ and $a_{0}>\sqrt{2n}$.

We can now use the discreteness assumption of $\s$ outside $K$ in order to find $b>a>a_{0}$ such that $D[a,b]\subseteq \reg$ and $D[a,b]$ does not contain any critical points of $f$. Indeed, consider $D(I)$, where $I\subseteq (a_{0},\infty)$ is any closed interval with $I^{\circ}\neq \emptyset$. Then since $D(I)$ is compact and $\s\cap D(a_{0},\infty)$ is discrete, the intersection $D(I)\cap \s$ is finite. If $s_{1}<\ldots < s_{k}$ are the distinct values assumed by $\varrho$ on these singular points, choose any closed interval $J\subseteq (s_{j},s_{j+1})$, so that $D(J)\subseteq \reg$. By Sard's lemma applied to $\varrho\vert_{\reg}$, there exists a regular value $\xi \in J^{\circ}$ and we have $D(\xi)\neq \emptyset$. It is easy to construct an open neighbourhood $U$ of the compact set $D(\xi)$ with $U\subseteq D(J)$ and $\nabla f\neq 0$ on $U$. At this point, using the compactness of $D(J)$, one can find $\varepsilon>0$ such that $D[\xi-\varepsilon,\xi+\varepsilon]\subseteq U$ and we conclude taking $a=\xi-\varepsilon$ and $b=\xi+\varepsilon$.

Fixed $D[a,b]$ as above, consider any end $\e\in \mathcal{E}(\X)$ and let $E=\e(D[0,a_{0}])$. We have $\varrho(E)=(a_{0},\infty)$ and in particular $E\cap D[a,b] \neq \emptyset$. Let $E[a,b]$ be a connected component of $D[a,b]$ contained in $E$ and, for every subinterval $J\subseteq [a,b]$, denote by $E(J)$ the unique connected component of $D(J)$ that is contained in $E[a,b]$.
\smallskip

\begin{claim}\label{cc} 
There exist an interval $(\alpha,\beta)\subseteq (a,b)$ and $i_{0}\in \N$ such that, for every $i\geq i_{0}$, there is a connected component $V_{i}$ of $D_{i}(\alpha,\beta)$ contained in $\phi_{i}(E(a,b))$.
\end{claim}
\begin{proof}
Let $\varepsilon<(b-a)/4$ and take $i_{0}\in \N$ such that $D[a,b]\subseteq U_{i}$ and $|\varrho-\phi_{i}^{\ast}\varrho_{i}|<\varepsilon$ on $D[a,b]$ for every $i\geq i_{0}$. A simple computation shows $\phi_{i}(E[a+3\varepsilon,b-3\varepsilon])\subseteq D_{i}(a+2\varepsilon,b-2\varepsilon)$. Let $V_{i}$ be the (unique) connected component of $D_{i}(a+2\varepsilon,b-2\varepsilon)$ that contains $\phi_{i}(E[a+3\varepsilon,b-3\varepsilon])$. Then $V_{i}\subseteq \phi_{i}(E(a,b))$. First of all $V_{i}$ intersects $\phi_{i}(E(a,b))$ by construction and we can take $x$ in this intersection. Suppose, by contradiction, that there exists $y\in V_{i}\setminus \phi_{i}(E(a,b))$ and let $\gamma\colon [0,1]\to V_{i}$ be a curve connecting $x$ to $y$. Now there exists $t_{\ast}\in (0,1]$ such that $\gamma(t)\in \phi_{i}(E(a,b))$ for $t\in [0,t_{\ast})$ and $\gamma(t_{\ast})\notin \phi_{i}(E(a,b))$. Considering the curve $\eta\colon [0,t_{\ast})\to E(a,b)$ defined as $\phi_{i}^{-1}\circ\gamma\vert_{[0,t_{\ast})}$ it is straightforward to check that $\varrho(\eta(t))\in (a+\varepsilon, b-\varepsilon)$ and therefore $\eta$ takes values in $E[a+\varepsilon,b-\varepsilon]$. Now let $(t_{j})\subseteq [0,t_{\ast})$ be a sequence of times converging to $t_{\ast}$. Since $\eta(t_{j})\in E[a+\varepsilon,b-\varepsilon]$, by compactness we get, up to a subsequence, $\eta(t_{j})\to z\in E[a+\varepsilon,b-\varepsilon]$. Therefore $\gamma(t_{j})=\phi_{i}(\eta(t_{j}))\to \phi_{i}(z)$ but on the other hand $\gamma(t_{j})\to \gamma(t_{\ast})$ by continuity. This gives $\gamma(t_{\ast})=\phi_{i}(z)$, contradicting the definition of $t_{\ast}$. The claim follows with $\alpha=a+2\varepsilon$ and $\beta=b-2\varepsilon$.
\end{proof}

Since $V_{i}$ is contained in a unique connected component of $D_{i}(\alpha,\infty)$ it defines a unique end $\e_{i}\in \mathcal{E}(M_{i})$ thanks to Proposition \ref{endspecialcase}. Note that in the above argument $a,b, \alpha, \beta, i_{0}$ are independent of $\e$. Therefore Claim \ref{cc} holds for every end of $\X$, with the same constants $a,b,\alpha,\beta,i_{0}$. In this way we get a map $\lambda_{i}\colon \mathcal{E}(\X)\to \mathcal{E}(M_{i})$ for every $i\geq i_{0}$. We claim that each $\lambda_{i}$ is injective and therefore $\#\mathcal{E}(\X)\leq \liminf_{i\to \infty}\#\mathcal{E}(M_{i})$.

Indeed, take $\e \neq \widetilde{\e}$ ends of $\X$. Then $E=\e(D[0,a_{0}])$ and $\widetilde{E}=\widetilde{\e}(D[0,a_{0}])$ are disjoint. 
In particular $E[a,b]\cap \widetilde{E}[a,b]=\emptyset$ and $V_{i}\cap \widetilde{V}_{i}=\emptyset$. Take the corresponding ends $\e_{i}$ and $\widetilde{\e}_{i}$ of $M_{i}$. If $\e_{i}=\widetilde{\e}_{i}$ then $E_{i}=\e_{i}(D[0,\alpha])=\widetilde{\e}_{i}(D[0,\alpha])$ and $V_{i},\widetilde{V}_{i}\subseteq E$. This contradicts Proposition \ref{prop.ends-shrinker} since $E_{i}$ is a connected component of $D_{i}(\alpha,\infty)$ without critical points. Therefore $\e_{i}\neq \widetilde{\e}_{i}$ and $\lambda_{i}$ is injective.

Finally, suppose that $\mathcal{E}(\X)$ is not finite. Then for every $m\in \N$ we can select a finite subset $\mathcal{F}_{m}\subseteq \mathcal{E}(\X)$ of cardinality $m$. As in the finite case we can find $a_{0}\geq r_{0}$ and constants  $a,b,\alpha,\beta,i_{0}$, this time possibly dependent on $m$, such that the natural map $\mathcal{E}(X)\to \pi_{0}(D(a_{0},\infty))$ is injective, $D[a,b]\subseteq \reg$ and Claim \ref{cc} holds for every end in $\mathcal{F}_{m}$. This again defines a map $\lambda_{i}\colon \mathcal{F}_{m}\to \mathcal{E}(M_{i})$ for $i\geq i_{0}$ and the same argument seen above proves that  $\lambda_{i}$ is injective. Therefore $\#\mathcal{E}(M_{i})\geq m$ for $i$ large. As $m$ was chosen arbitrarily, we conclude that $\liminf_{i\to \infty}\#\mathcal{E}(M_{i})=\infty=\#\mathcal{E}(\X)$.

\item Due to the first point, we only need to prove the inequality $\limsup_{i\to \infty}\#\mathcal{E}(M_{i})\leq \#\mathcal{E}(\X)$.
Let $m=\limsup_{i\to \infty}\#\mathcal{E}(M_{i})<\infty$. If $m=0$, there is nothing to prove. If $m\geq 1$, up to a subsequence we can assume $\#\mathcal{E}(M_{i})=m$ for every $i$. In particular each $M_{i}$ is not compact and $D_{i}(a,\infty)\neq \emptyset$ for every $a>0$. Thanks to the convergence, $\X$ is non-compact too.

By assumption, $\s$ is contained in a compact subset and therefore we can find $a_{0}\geq r_{0}$ such that $\s \subseteq D[0,a_{0}]$. Thanks to Proposition \ref{endspecialcase}, we know that $\X$ has a finite number of ends and the natural map 
$\mathcal{E}(\X)\to \mathcal{E}(\X, D[0,a_{0}])$ is a bijection, where we can identify $\mathcal{E}(\X, D[0,a_{0}])$ and $\pi_{0}(D(a_{0},\infty))$. Now consider an end $\e_{i}$ of $M_{i}$ and let $E_{i}=\e_i(D[0,a_{0}])$. Moreover if $J\subseteq (a_{0},\infty)$ is an interval, denote with $E_{i}(J)$ the unique connected component of $D(J)$ contained in $E_i$. Let $b>a>a_{0}+(10n+1)$. We have a result similar to Claim \ref{cc}.
\smallskip

\begin{claim}\label{cc2} 
There exist $i_{0}\in \N$ such that, for every $i\geq i_{0}$, we have $E_{i}[a,b] \subseteq \phi_{i}(U_{i})$. If we define $W_{i}=\phi_{i}^{-1}(E_{i}[a,b])\subseteq \X$ then there exists and interval $(\alpha,\beta)\subseteq (a,b)$ such that $W_{i}$ contains a connected component of $D(\alpha,\beta)$.
\end{claim}

\begin{proof}
Let $\varepsilon<\min\{(b-a)/4, 1/(b+6n)\}$ and take $i_{0}\in \N$ such that $D[a-c,b+c]\subseteq U_{i}$ and $|\phi_{i}^{\ast}\varrho_{i}-\varrho|<\varepsilon$ on $D[a-c,b+c]$ for every $i\geq i_{0}$, where $c=10n+1$. We can also suppose that every $\widehat{\phi}_{i}$ is an $\varepsilon_{i}$-pGHA with $\varepsilon_{i}<\varepsilon$. By the choice of $\varepsilon$ and using \eqref{rhogrowth}, one can check that $E_{i}[a,b]\subseteq B(p_{i},1/\varepsilon_{i})$. Therefore, if $x\in E_{i}[a+\varepsilon,b-\varepsilon]$ there exists $y\in B(p,1/\varepsilon_{i})\subseteq \X$ such that $d(x,\widehat{\phi}_{i}(y))<\varepsilon_{i}$ and an easy computation gives $|d(y,p)-d(x,p_{i})|<2\varepsilon_{i}$. Again by \eqref{rhogrowth} we deduce $\varrho(y)\in (a-c,b+c)$ and therefore $y \in D[a-c,b+c]\subseteq U_{i}\subseteq \reg$. Now observe that if $z\in B(x,\varepsilon)$ then $|\varrho_{i}(x)-\varrho_{i}(z)|\leq d(x,z)<\varepsilon$ and therefore $B(x,\varepsilon)\subseteq E_{i}[a,b]$ by connectedness. In particular $\phi_{i}(y)\in E_{i}[a,b]$.
Moreover, we have $a-1<\varrho(y)<b+1$, that is, $y \in D(a-1,b+1)$, and therefore $\phi_{i}(D(a-1,b+1))\cap E_{i}[a,b]\neq \emptyset$.

Now we show that the inclusion $E_{i}[a,b]\subseteq \phi_{i}(D(a-1,b+1))$ holds. Suppose by contradiction that there exists $y\in E_{i}[a,b]$ not in $\phi_{i}(D(a-1,b+1))$ and let $\gamma\colon [0,1]\to E_{i}[a,b]$ be a curve connecting $x\in E_{i}[a,b]\cap \phi_{i}(D(a-1,b+1))$ to $y$. Since $\phi_{i}(D(a-1,b+1))\subseteq M_{i}$ is open, there exists $t_{\ast}\in (0,1]$ such that $\gamma(t)\in \phi_{i}(D(a-1,b+1))$ for $t\in [0,t_{j})$ but $\gamma(t_{\ast})\notin \phi_{i}(D(a-1,b+1))$. Let $\eta\colon [0,t_{\ast})\to D(a-1,b+1)$ be the curve obtained composing $\gamma\vert_{[0,t_{j})}$ and $\phi_{i}^{-1}$. If $(t_{j})\subseteq [0,t_{\ast})$ is a sequence of times converging to $t_{\ast}$ we have $\eta(t_{j})\in D(a-\varepsilon, b+\varepsilon)$. Up to a subsequence, we have $\eta(t_{j})\to z \in D[a-\varepsilon,b+\varepsilon]$ and therefore $\gamma(t_{j})=\phi_{i}(\eta(t_{j}))\to \phi_{i}(z)$. On the other hand, by continuity $\gamma(t_{j})\to \gamma(t_{\ast})$ and we obtain $\gamma(t_{\ast})=\phi_{i}(z)$, in contradiction with the assumption.

Consider $W_{i}=\phi_{i}^{-1}(E_{i}[a,b])\subseteq \X$. Being connected, $W_{i}$ is contained in a unique connected component of $D[a-1,b+1]$ and we can consider the unique connected component of $D[a+2\varepsilon,b-2\varepsilon]$ contained in the latter, which we denote $E[a+2\varepsilon,b-2\varepsilon]$. We claim $E[a+2\varepsilon,b-2\varepsilon] \subseteq W_{i}$. Indeed take $x\in W_{i}\cap E[a+2\varepsilon,b-2\varepsilon]$, which is not empty. Then $x=\phi_{i}^{-1}(y)$ for some $y\in E_{i}[a,b]$. This gives $E_{i}[a,b]\cap \phi_{i}(E[a+2\varepsilon,b-2\varepsilon])\neq \emptyset$. Therefore $\phi_{i}(E[a+2\varepsilon,b-2\varepsilon])\subseteq E_{i}[a,b]$ being contained in $D[a,b]$. Finally $E[a+2\varepsilon,b-2\varepsilon]\subseteq W_{i}$ and we conclude setting $\alpha=a+2\varepsilon$ and $\beta=b-2\varepsilon$. This finishes the proof of the claim.
\end{proof}

Now $W_{i}$ is connected and is contained in a unique connected component of $D(a_{0},\infty)$, determining a unique end $\e^{\infty}_{i}\in \mathcal{E}(\X)$ thanks to Proposition \ref{endspecialcase}. Note that in the above argument all the constants $a,b,\alpha,\beta,i_{0}$ are independent of the specific end $\e_{i}$. Therefore Claim \ref{cc2} holds for every end of $M_{i}$ where $i\geq i_{0}$ and with the same constants $a,b,\alpha,\beta,i_{0}$. This produces a map $\omega_{i}\colon \mathcal{E}(M_{i})\to \mathcal{E}(\X)$ for every $i\geq i_{0}$. Every $\omega_{i}$ is injective and therefore $\limsup_{i\to \infty}\#\mathcal{E}(M_{i})\leq \#\mathcal{E}(\X)$. Indeed, take $\e\neq \widetilde{\e}$ ends of $M_{i}$ and let $\e^{\infty}_{i}$, $\widetilde{\e}^{\infty}_{i}$ be the corresponding ends of $\X$. Since $\e(D[0,a])\cap \widetilde{\e}(D[0,a])=\emptyset$, we have $E_{i}(a,b)\cap \widetilde{E}_{i}(a,b)=\emptyset$ and therefore $W_{i}, \widetilde{W}_{i}$ are disjoint. In particular, the connected components of $D(\alpha,\beta)$ that they contain are disjoint, giving $\e^{\infty}_{i}\neq \widetilde{\e}^{\infty}_{i}$.

Finally, suppose $\limsup_{i\to \infty}\#\mathcal{E}(M_{i})=\infty$. Then given $m\in \N$ we can suppose, up to a subsequence, that $\#\mathcal{E}(M_{i})\geq m$ for every $i$ large. As $\mathcal{E}(M_{i})$ is finite, the preceding argument applies verbatim to the case in consideration, remarking again that $a,b,\alpha,\beta,i_{0}$ are constants independent of the ends of $M_{i}$ (but now possibly dependent on $m$). Therefore we obtain $\#\mathcal{E}(\X)\geq \#\mathcal{E}(M_{i})\geq m$ and since $m$ was chosen arbitrarily we deduce $\#\mathcal{E}(\X)=\infty=\limsup_{i\to \infty}\#\mathcal{E}(M_{i})$. The proof is complete.\qedhere
\end{enumerate}
\end{proof}

\begin{rem} 
As a consequence of this theorem, if $\#\mathcal{E}(\X)=\infty$ then the sequence $\#\mathcal{E}(M_{i})$ diverges. This cannot happen if we add the condition $\mu_{i}\geq \underline{\mu}$ and in this case we deduce $\#\mathcal{E}(\X)\leq \Lambda(n,\underline{\mu},r_{0})$ thanks to Theorem \ref{thm.endbound}. Under the assumption of the second point of the theorem, $\#\mathcal{E}(\X)$ is always finite (see Remark \ref{rem.11singular}) and therefore the number of ends $\#\mathcal{E}(M_{i})$ stabilises from some point onwards along the sequence.
\end{rem}

\section{Applications to moduli spaces}\label{applications-convergence}

The preceding results can be used to derive some information regarding the structure of moduli spaces of Ricci shrinkers. 
We recall the following compactness results.

\begin{prop}[Theorem 1.1 and Theorem 8.6 in \cite{LLW}; Theorem 1.1 and Theorem 2.5 in \cite{hm}; Theorem 1.1 in \cite{hm4}] \label{LLWcompactness} 
Let $(M_{i},g_{i},f_{i},p_{i})_{i\in \N}$ be a sequence of $n$-dimensional Ricci shrinkers with basepoint $p_{i}\in M_{i}$ and entropy bounded below $\mu_{i}\geq \underline{\mu}$. Then a subsequence of $(M_{i},g_{i},f_{i},p_{i})_{i \in \N}$ converges, in the pointed singular Cheeger-Gromov sense, to a $n$-dimensional singular Ricci shrinker $\X$ with singular part of Minkowski codimension at least four, $\dim_{\mathcal{M}} \s\leq n-4$.

When $n=4$ or under the local energy condition 
\begin{equation}\label{energycondition}
\int_{B(p_{i},r)}|\Rm(g_{i})|^{n/2}_{g_{i}}d\nu_{g_{i}}\leq E(r)<\infty, \quad \forall i,r,
\end{equation}
then the limit is an orbifold Ricci shrinker, in particular $\s$ is discrete and $f$ is critical at every point in $\s$.

Finally, if we assume a lower curvature bound $\Rm(g_{i}) \geq \underline{K}$, then we have convergence to a smooth Ricci shrinker in the pointed smooth Cheeger-Gromov sense.
\end{prop}

\begin{rem} We point out that condition $\Rm(g_{i})\geq \underline{K}$ actually implies the energy condition \eqref{energycondition}.
\end{rem}

We are going to work with various moduli spaces of Ricci shrinkers. We will always assume that the shrinkers under consideration are $n$-dimensional, with basepoint $p$, entropy bounded below $\mu(g)\geq \underline{\mu}$, and energy condition \eqref{energycondition} if $n\geq 5$.

First of all, denote by $\m[n,\underline{\mu},E(r),r_{0}]$ the moduli space formed by the Ricci shrinkers $M$ satisfying the above conditions as well as $|\nabla{f}|>0$ on $D(r_{0},\infty)$. In the following, all moduli spaces $\m$ of our interest will satisfy $\m\subseteq \m[n,\underline{\mu},E(r),r_{0}]$ for some appropriate $r_{0}>0$ and therefore Proposition \ref{thm.endbound} implies that $\#\mathcal{E}(M)\leq \Lambda=\Lambda(n,\underline{\mu},r_{0})$ for every $M\in \m$. We can therefore write 
\begin{equation*}
\m=\bigsqcup_{N\leq \Lambda}\m_{N},
\end{equation*}
where $\m_{N}$ denotes the moduli sub-space formed by those shrinkers in $\m$ with \emph{exactly} $N$ ends. Starting with $\m_N[n,\underline{\mu},E(r),r_{0}]$, a direct application of Theorem \ref{thm.limitendbound} and Theorem \ref{LLWcompactness} gives the following.

\begin{cor}\label{cor1}
The boundary points of the moduli space $\m_N[n,\underline{\mu},E(r),r_{0}]$ are limit orbifold Ricci shrinkers with at most $N$ ends.
\end{cor}

Now suppose that we have a scalar curvature bound. In this case it is possible to say something more. Consider the moduli space $\m[n,\underline{\mu},E(r),A]$ of $n$-dimensional Ricci shrinkers with entropy bounded below $\mu(g)\geq \underline{\mu}$, energy condition \eqref{energycondition} (if $n\geq 5$), and bounded scalar curvature $S\leq A$. We have the following result.

\begin{cor}\label{cor2}
The boundary points of the moduli space $\m_N[n,\underline{\mu},E(r),A]$ are given by limit orbifold Ricci shrinkers with singular locus contained in a compact set and exactly $N$ ends.
\end{cor}
\begin{proof}
Let $(M_{i},g_{i},f_{i},p_{i})_{i \in \N}$ be a sequence in $\m_N[n,\underline{\mu},E(r),A]$. Up to a subsequence, thanks to Theorem \ref{LLWcompactness}, we have convergence to an orbifold Ricci shrinker $\X$. Observe that condition $S_{i}\leq A$ passes to the limit, giving $S\leq A$ on $\reg\subseteq \X$. From this and \eqref{auxilliaryconstant}, \eqref{fgrowth} it follows that there exists $r_{0}>0$ such that $|\nabla f_{i}|\geq c_0>0$ on $D_{i}(r_{0},\infty)$ for every $i \in \N$ and $|\nabla f|\geq c_0>0$ on $D(r_{0},\infty)\cap \reg$. The lower bound $|\nabla f_{i}|\geq c_0>0$ implies that the limit $f$ cannot have critical points on $D(r_{0},\infty)$ and by Theorem \ref{LLWcompactness} it follows that $\s \subseteq \X$ is contained in $D[0,r_0]$. Therefore we can apply Point 2 of Theorem \ref{thm.limitendbound}, obtaining $\#\mathcal{E}(\X)=N$.  
\end{proof}

Under additional conditions it is possible to derive some compactness results. Let $\m[n,\underline{\mu},\underline{K},r_{0},c_0]$ be the moduli space of Ricci shrinkers with curvature bounded below $\Rm\geq \underline{K}>-\infty$ and with $|\nabla f|\geq c_0>0$ on $D(r_{0},\infty)$. We claim that every $\m_N[n,\underline{\mu},\underline{K},r_{0},c_0]$ is compact. 

\begin{cor}\label{cor3} 
The moduli space $\m_N[n,\underline{\mu},\underline{K},r_{0},c_0]$ is compact with respect to the pointed smooth Cheeger-Gromov convergence.
\end{cor}
\begin{proof}
If $(M_{i},g_{i},f_{i},p_{i})_{i\in \N}$ is a sequence in $\m_N[n,\underline{\mu},\underline{K},r_{0},c_0]$, thanks to Theorem \ref{LLWcompactness} we have convergence, up to a subsequence, to a smooth Ricci shrinker $(M,g,f,p)$ in the pointed smooth Cheeger-Gromov sense. Hence $M$ satisfies the estimates $\mu(g)\geq \underline{\mu}$, $\Rm_{g}\geq \underline{K}$, and $|\nabla f|\geq c_0>0$ on $D(r_{0},\infty)$. At this point we can invoke Point 2 of Theorem \ref{thm.limitendbound} obtaining $\#\mathcal{E}(M)=N$. This yields $M\in \m_N[n,\underline{\mu},\underline{K},r_{0},c_0]$, and the proof is complete. 
\end{proof}

In particular, also the subspace $\m_N[n,\underline{\mu},\underline{K},A]$ of $n$-dimensional Ricci shrinkers with lower entropy bound $\mu(g)\geq \underline{\mu}$, lower curvature bound $\Rm\geq \underline{K}$, upper scalar curvature bound $S \leq A$, and exactly $N$ ends is compact with respect to pointed smooth Cheeger-Gromov convergence.

Finally, in dimension $n=4$, we can use the following result of Munteanu-Wang.

\begin{prop}[Theorem 2.4 in \cite{mw15}]\label{mw-primo} 
Let $M$ be a $4$-dimensional Ricci shrinker with bounded scalar curvature $S \leq A$. Then there exists $r_{1}>0$,
depending only on $A$, such that 
\[
\sup_{M} |\Rm| \leq C
\]
where $C>0$ is a constant depending on $A$ and on $\sup_{D[0,r_{1})} |\Rm|$.
\end{prop}

Setting $\m[4,\underline{\mu}, A,B]$ to be the moduli space of $4$-dimensional Ricci shrinkers with entropy lower bound, bounded scalar curvature $S\leq A$ and $|\Rm|\leq B$ on $D[0,r_{1})$, where $r_{1}>0$ is as in Theorem \ref{mw-primo} (in particular depending on $A$), then Theorem \ref{mw-primo} gives a global bound $|\Rm|\leq C$ on $M$ for every element of this moduli space, where $C=C(A,B)>0$. In particular, we can use the above smooth compactness result to conclude the following.

\begin{cor}\label{cor4}
The moduli space $\m_N[4,\underline{\mu}, A,B]$ is compact with respect to the pointed smooth Cheeger-Gromov convergence.  
\end{cor}

\section{Asymptotically conical ends and nearly conical regions}\label{section-conicalends}

The main goal of this section is to study curvature and volume estimates on asymptotically conical ends and nearly conical regions as in Definitions \ref{asymptconical}  and \ref{nearannulus}, respectively. We start by recalling the definition and basic properties of a Riemannian cone.

\begin{defn} 
Let $(\Sigma,g_{\Sigma})$ be a connected, compact Riemannian manifold. The \emph{Riemannian cone} over $\Sigma$ is the product manifold $\mathcal{C}_{\Sigma}=(0,\infty)\times \Sigma$ endowed with the warped metric 
\begin{equation*}
g_{c}=d\r^{2}+\r^{2}g_{\Sigma}.
\end{equation*}
When $I\subseteq (0,\infty)$ is an interval, $\mathcal{C}_{\Sigma}(I)=I\times \Sigma$ endowed with the cone metric $g_{c}$ is called a \emph{conical annulus}. We denote points in the cone by $x=(\r,y)$ and $\r\colon \mathcal{C}_{\Sigma}\to (0,\infty)$ is the \emph{radial coordinate}. 
\end{defn}

\begin{lemma}\label{lemmacono} 
Consider $\mathcal{C}_{\Sigma}(a,\infty)$ with its intrinsic Riemannian distance $d_{c}$, where $a\geq 0$.
\begin{enumerate}
\item Fix a point $o \in \mathcal{C}_{\Sigma}(a,\infty)$. Then
\begin{equation}\label{stimadistcono}
|\r(x)-\r(o)|\leq d_{c}(x,o)\leq \r(o)+\r(x)-2a+a\diam(\Sigma)
\end{equation}
for every $x\in \mathcal{C}_{\Sigma}(a,\infty)$. In particular 
\begin{equation*}
\lim_{\r(x)\to \infty}\frac{d_{c}(o,x)}{\r(x)}=1.
\end{equation*}
\item A subset $\Omega \subseteq \mathcal{C}_{\Sigma}(a,\infty)$ is bounded if and only if $\sup_{x\in \Omega}\r(x)$ is finite. In this case, $\mathcal{C}_{\Sigma}(a,\infty)\setminus \Omega$ has only one unbounded connected component.
\end{enumerate}
\end{lemma}
\newpage

\begin{proof}\
\begin{enumerate}
\item Let $\gamma\colon [0,1]\to \mathcal{C}_{\Sigma}(a,\infty)$ be any piecewise smooth curve connecting $o$ to $x$. Writing $\gamma(t)=(\r(t),\sigma(t))$, where $\r\colon [0,1]\to (a,\infty)$ and $\sigma\colon [0,1]\to \Sigma$, we have
\begin{equation}\label{stimacurvacono}
\begin{aligned}
|\r(x)-\r(o)|&\leq \int_{0}^{1}|\dot{\r}(t)|dt \leq L_{g_{c}}(\gamma)=\int_{0}^{1}\sqrt{|\dot{\r}(t)|^{2}+\r(t)^{2}|\dot{\sigma}(t)|^{2}}dt \\
&\leq \int_{0}^{1}|\dot{\r}(t)|dt +\int_{0}^{1}|\r(t)||\dot{\sigma}(t)|dt
\end{aligned}
\end{equation}
and the lower bound in \eqref{stimadistcono} follows. To prove the upper bound, suppose that $\sigma\colon [0,1]\to \Sigma$ is a minimising geodesic from $y(o)$ to $y(x)$. Then 
\begin{equation*}
\int_{0}^{1}|\r(t)||\dot{\sigma}(t)|dt =\bigg(\int_{0}^{1}|\r(t)|dt\bigg)d_{\Sigma}(y(o),y(x)) \leq \bigg(\int_{0}^{1}|\r(t)|dt\bigg)\diam(\Sigma). 
\end{equation*}
Take $a<b<\min\{\r(o),\r(x) \}$ and $\varepsilon\in (0,1)$. If we define 
\begin{equation*}
\r(t)=\begin{cases}
\r(o)+(b-\r(o))t/\varepsilon, & \text{if } t\in [0,\varepsilon],\\
b, & \text{if } t\in [\varepsilon,1-\varepsilon],\\
\r(x)+(b-\r(x))(1-t)/\varepsilon, & \text{if } t\in [1-\varepsilon,1],
\end{cases}
\end{equation*}
we obtain 
\begin{equation*}
\int_{0}^{1}|\dot{\r}(t)|dt=\r(o)+\r(x) -2b \quad \text{and} \quad  \int_{0}^{1}|\r(t)|dt=b+\frac{1}{2}\varepsilon(\r(o)-b)+
\frac{1}{2}\varepsilon(\r(x)-b).
\end{equation*}
Substituting in \eqref{stimacurvacono} and taking the limit as $(b,\varepsilon)\to (a,0)$, we obtain the claimed upper bound.

\item The first part follows directly from the preceding point. Moreover, if $\Omega$ is bounded, we have $\Omega\subseteq \mathcal{C}_{\Sigma}(a,b]$ for some $b>a$. Now $\mathcal{C}_{\Sigma}(b,\infty)$ is contained in a unique connected component of $\mathcal{C}_{\Sigma}(a,\infty)\setminus \Omega$, which therefore is unbounded. All other connected components of $\mathcal{C}_{\Sigma}(a,\infty)\setminus \Omega$ are contained in $\mathcal{C}_{\Sigma}(a,b]$ and therefore are bounded.\qedhere
\end{enumerate}
\end{proof}
 
The preceding proposition has some direct consequences for asymptotically conical ends as in Definition \ref{asymptconical}. Note that, given such an asymptotically conical end $E$, we can consider the diffeomorphism $\varphi\colon \mathcal{C}_{\Sigma}(a,\infty)\to E$ from Definition \ref{asymptconical} and write $\r(x)=\r(\varphi^{-1}(x))$ for $x\in E$, thinking of $\r\colon E\to (a,\infty)$ as a radial coordinate. We then obtain the following.

\begin{prop}\label{propconicalend} 
Let $E\subseteq \X$ be a $C^{0}$-asymptotically conical end relative to some compact subset $K$.
\begin{enumerate}
\item Fix a point $o\in E$. If $x\in E$ then 
\begin{equation*}
d(o,x)\to \infty \Longleftrightarrow d_{E}(o,x)\to \infty \Longleftrightarrow \r(x)\to \infty,
\end{equation*}
where $d_{E}$ is the intrinsic Riemannian distance on $E$.
\item $\Omega \subseteq E$ is bounded in $\X$ if and only if it is bounded in $E$ with respect to the intrinsic Riemannian distance.
\item Let $H\supseteq K$ be compact. Then $E\setminus H$ has only one unbounded connected component. In other words there is only one end of $\X$ relative to $H$ contained in $E$.
\item If $b>a$ then $F=\varphi(\mathcal{C}_{\Sigma}(b,\infty))$ is an end of $\X$ relative to the compact subset $K\cup \varphi(\mathcal{C}_{\Sigma}(a,b])$ and it is asymptotically conical. 
More precisely it satisfies Definition \ref{asymptconical} via the diffeomorphism 
$\varphi\vert_{\mathcal{C}_{\Sigma}(b,\infty)}\colon \mathcal{C}_{\Sigma}(b,\infty)\to F$.
\item We have for all $x\in E$
\begin{equation}\label{radialdistance}
\lim_{\r(x) \to \infty}\frac{d(o,x)}{\r(x)}=1.
\end{equation}
The same is true with respect to a general basepoint $p\in M$:
\begin{equation}\label{radialdistancebasepoint}
\lim_{\r(x)\to \infty}\frac{d(p,x)}{\r(x)}=1.
\end{equation}
\end{enumerate}
\end{prop}

\begin{proof}
Recall that by Definition \ref{singularRS} the intrinsic Riemannian metric $d_g$ on $\reg$ induced by $g$ coincides with the restriction $d\vert_{\reg}$. Now denote $d_{c}$ the intrinsic Riemannian distance on the cone $\mathcal{C}_{\Sigma}(a,\infty)$ and let 
$\psi=\varphi^{-1}$. We have $d_{E}\geq d_g\equiv d\vert_{\reg}$.

\begin{enumerate}
\item From \eqref{conicalconditionmetric} it follows that there exist a constant $c>0$ such that $\varphi^{\ast}g \leq c^{2}g_{c}$ 
on the whole $\mathcal{C}_{\Sigma}(a,\infty)$. In particular from \eqref{stimadistcono} we have
\begin{equation*}
d(x,o) \leq d_{E}(x,o) \leq c d_{c}(\psi(x),\psi(o))\leq c \big(r(o)+ \r(x) -2a +a\diam(\Sigma)\big),
\end{equation*}
for every $x\in E$. This says that if $d(x,o)\to \infty$ then $d_{E}(x,o)\to \infty$ and therefore $\r(x)\to \infty$.

Conversely, using \eqref{conicalcondition} it is easy to see that there exists $b>a$ such that $\varphi^{\ast}g \geq \frac{1}{2} g_{c}$ on $\mathcal{C}_{\Sigma}(b,\infty)$. Now take $x\in E$ with $\r(x)>b$ and consider a curve $\gamma\colon [0,1]\to \reg$ from $x$ to $o$. If 
$\gamma$ is entirely contained in $E$ put $t_{\ast}=1$, otherwise there exists $t_{\ast}\in [0,1)$ such that $\gamma(t)\in E$ for $t\in [0,t_{\ast})$ but $\gamma(t_{\ast})\notin E$. In every case consider $\eta\colon [0,t_{\ast})\to C_{\Sigma}(a,\infty)$, where $\eta=\psi\circ \gamma\vert_{[0,t_{\ast})}$. Then $\eta$ exits from every compact $H\subseteq \mathcal{C}_{\Sigma}(a,\infty)$ not containing $\psi(o)$. If this is not the case there exists $H$ compact with $\psi(o)\notin H$ and a sequence of times $(t_{k})\subseteq [0,t_{\ast})$ converging to $t_{\ast}$ such that $\eta(t_{k})\in H$ for every $k$. Therefore $\gamma(t_{k})$ belongs to the compact set $\varphi(H)\subseteq E$. It follows 
$\gamma(t_{k})\to \gamma(t_{\ast})\in \varphi(H)$, in contradiction to the definition of $t_{\ast}$ in both cases.

Consider the compact subset $H=[\r(x)-\delta,\r(x)+\delta]\times \Sigma$, where $\delta < \delta_{0}=\min\{|\r(x)-\r(o)|, \r(x)-b\}$. Then $H\subseteq \mathcal{C}_{\Sigma}(b,\infty)$, $\eta(0)\in H$ and $\psi(o)\notin H$. Thanks to the above argument there exists $t_{+}\in (0,t_{\ast})$ such that $\eta(t)\in H$ for every $t \in [0,t_{+}]$ and $\eta(t_{+})\in \partial H$. Then, writing $\eta(t)=(\r(t),x(t))$
\begin{equation*}
L_{g_{c}}(\eta\vert_{[0,t_{+}]})=\int_{0}^{t_{+}}\sqrt{\dot{\r}(t)^{2}+\r(t)^{2}\dot{x}(t)^{2}}dt \geq \int_{0}^{t_{+}}|\dot{\r}(t)|dt
\geq \delta
\end{equation*}
since $\r(t)$ connects $\r(x)$ to $\r(\eta(t_{+}))\in \{\r(x)-\delta,\r(x)+\delta\}$. Now 
\begin{equation*}
L_{g}(\gamma)\geq L_{g}(\gamma\vert_{[0,t_{+}]}) \geq \frac{1}{2} L_{g_{c}}(\eta\vert_{[0,t_{+}]})\geq  \frac{\delta}{2}
\end{equation*} 
and we conclude that $d(x,o)\geq \frac{\delta}{2}$. This can be improved to $d(x,o)\geq \frac{\delta_{0}}{2}$, as $\delta<\delta_{0}$ was chosen arbitrarily. Summing up, we finally have
\begin{equation*}
\frac{1}{2}\min\{\r(x)-b, |\r(x)-\r(0)|\}\leq d(x,o)
\end{equation*}
for every $x\in E$ with $\r(x)>b$. In conclusion, if $\r(x)\to \infty$ the above inequality implies $d(x,o)\to \infty$, and we are done.

\item If $\Omega$ is bounded with respect to $d_{E}$ then it is bounded in $\reg$ and in $\X$. Conversely, suppose towards a contradiction that $\Omega$ is bounded in $\X$ but not in $E$. Then there exists a sequence $(x_{i})\subseteq \Omega$ with $d_{E}(o,x_{i})\to \infty$. In particular $\r(x_{i})\to \infty$. Thanks to the first point above, we have $d(o,x_{i})\to \infty$, which contradicts the boundedness of $\Omega$ in $\X$.

\item The set $E\cap H$ is bounded in $\X$ and therefore in $E$. Since $\psi(E\cap H)$ is bounded in $\mathcal{C}_{\Sigma}(a,\infty)$ the complement $\mathcal{C}_{\Sigma}(a,\infty)\setminus \psi(E\cap H)$ has only one unbounded connected component $F$. Now $\varphi(F)$ is the only unbounded connected component of $E\setminus H$.

\item Denote $\Omega=\mathcal{C}_{\Sigma}(a,b]$ for short. Then $\Omega$ is a bounded subset of the cone and therefore $\varphi(\Omega)\subseteq E$ is bounded in $E$ and in $\X$, thanks to Point 2. By definition, $F$ is connected, unbounded and open, being $\Omega$ closed in $\mathcal{C}_{\Sigma}(a,\infty)$. In particular $\overline{\varphi(\Omega)}\cap E=\varphi(\Omega)$. Let $H=K\cup \overline{\varphi(\Omega)}$. Then $H$ is compact and $E\setminus H=E\setminus (E\cap \overline{\varphi(\Omega)})=E\setminus \varphi(\Omega)=F$. This proves that $F$ is a connected component of 
$\X \setminus H$ and therefore an end. Finally, the restriction $\varphi\vert_{\mathcal{C}_{\Sigma}(b,\infty)}\colon \mathcal{C}_{\Sigma}(b,\infty)\to F$ satisfies conditions \eqref{conicalconditionmetric}, \eqref{conicalcondition} and we conclude.

\item As seen in Point 1, we have 
\begin{equation}\label{stimadistconica}
d(x,o) \leq  c \big(\r(o)+\r(x)-2a +a\diam(\Sigma)\big)
\end{equation}
for every $x\in E$. For $b\geq a$, introduce the function
\begin{equation*}
\sigma(b)=\sup_{x\in \mathcal{C}_{\Sigma}(b,\infty)}|\varphi^{\ast}g-g_{c}|(x).
\end{equation*}
Then $\sigma$ is non-increasing and condition \eqref{conicalcondition} gives $b_{\ast}>a+1$ such that $\sigma(b_{\ast})<1/2$. Given a point $x\in E$ with $\r(x)>b_{\ast}^{2}$ denote $s(x)=\sqrt{\r(x)}$. Then $s(x)>b_{\ast}$ and $\r(x)>s(x)$. In particular $\sigma(s(x))\leq \sigma(b_{\ast})<1/2$ and $\psi(x)\in \mathcal{C}_{\Sigma}(s(x),\infty)$. Now take $o_{x}\in E$ with $\r(o_{x})=2s(x)$. If we define $F_{x}=\varphi(\mathcal{C}_{\Sigma}(s(x),\infty))$ then $F_{x}\subseteq E$ is a $C^{0}$-asymptotically conical end of $E$, thanks to Point 4 above, and $x,o_{x}\in F_{x}$. In order to estimate $d(o,x)$ write $d(o_{x},x)-d(o,o_{x}) \leq d(x,o)\leq d(x,o_{x})+d(o,o_{x})$. Applying \eqref{stimadistconica} to $o, o_{x}\in E$ we have
\begin{equation*}
d(o,o_{x})\leq c \big(\r(o)+2s(x)-2a +a\diam(\Sigma)\big).
\end{equation*}
Now we work on the end $F_{x}$. From \eqref{conicalcondition}, we obtain
\begin{equation*}
|\varphi^{\ast}g-g_{c}|(y) \leq \sup_{z\in \mathcal{C}_{\Sigma}(s(x),\infty)}|\varphi^{\ast}g-g_{c}|(z)=\sigma(s(x))<1/2
\end{equation*}
for every $y \in \mathcal{C}_{\Sigma}(s(x),\infty)$ and therefore
\begin{equation*}
(1-\sigma(s(x)))g_{c}\leq \varphi^{\ast}g\leq (1+\sigma(s(x)))g_{c}.
\end{equation*} 
Arguing as in Point 1, but this time for $o_{x}, x \in F_{x}$, we get
\begin{align*}
d(o_{x},x) &\leq \sqrt{1+\sigma(s(x))} \big(\r(o_{x})+\r(x)-2s(x) +s(x)\diam(\Sigma) \big) \\
&=\sqrt{1+\sigma(s(x))}\big(\r(x)+s(x)\diam(\Sigma) \big)
\end{align*}
and similarly 
\begin{equation*}
d(x,o_{x})\geq \sqrt{1-\sigma(s(x))}\min\big\{\r(x)-s(x), |\r(x)-2s(x)|\big\}.
\end{equation*}
Since the above estimates hold for every $x\in E$ with $\r(x)>b_{\ast}^{2}$ it easily follows 
\begin{equation*}
\lim_{\r(x)\to \infty}\frac{d(o,x)}{\r(x)}=1.
\end{equation*}
The last assertion follows from the triangle inequality.\qedhere
\end{enumerate}
\end{proof}

Now we want to see how curvature behaves along a conical end. We start with a simple observation. Let $E\subseteq \X$ be a $C^{k}$-asymptotically conical end and let $\varphi\colon \mathcal{C}_{\Sigma}(a,\infty)\to E$ be as in Definition \ref{asymptconical}. Then for every $\varepsilon>0$ there exists $b>a$ such that 
\begin{equation}\label{stimadiffmetrichecono}
\vert \nabla^{\ell}(\varphi^{\ast}g-g_{c})\vert(x) \leq \varepsilon \r(x)^{-\ell} 
\end{equation}
for every $x$ with $\r(x)>b$ and for $\ell=0,\ldots,k$. Condition \eqref{stimadiffmetrichecono} allows us to estimate the curvature using the following result, which is useful to state in a quite general form.

\begin{prop}\label{closenessmetric} 
Let $\Omega\subseteq \X$ be an open subset with $\Omega\subseteq \reg$ and suppose there exists a diffeomorphism $\varphi\colon \mathcal{C}_{\Sigma}(a,b)\to \Omega$ such that 
\begin{equation*}
|\nabla^{\ell}(\varphi^{\ast}g-g_{c})|(x)\leq \varepsilon_{\ell}(x)
\end{equation*}
for every $0\leq \ell \leq 2$, where $0<a<b\leq \infty$ and $\varepsilon_{\ell}\colon \mathcal{C}_{\Sigma}(a,b)\to (0,1/2)$ are functions. Then 
\begin{equation}\label{curvaturecondition} 
\vert \Rm(\varphi^{\ast}g)\vert_{\varphi^{\ast}g}(x) \leq C(\r(x)^{-2}+\varepsilon_{1}^{2}(x)+\varepsilon_{2}(x)),
\end{equation}
for every $x\in \Omega$, where $C>0$ is a constant depending only on $n$ and $\Sigma$. 
\end{prop}

\begin{proof}
Denote by $h=\varphi^{\ast}g$ the pull-back metric on $\mathcal{C}_{\Sigma}(a,b)$. From Propositions \ref{nearmetrics} and \ref{curvaturametrichevicine}, we find for some $C>0$ depending only on $n$
\begin{equation*}
|\Rm(h)|_{h}\leq C |\Rm(h)|_{g_{c}}\leq C \big(|\Rm(g_{c})|_{g_{c}}+\varepsilon_{1}^{2}+\varepsilon_{2}\big)
\end{equation*}
and therefore it is sufficient to bound $\Rm(g_{c})$. To this end, denote by $\Riem$ the $(0,4)$-version of $\Rm$ and recall that 
$|\Rm|=|\Riem|$. We have $\Riem(g_{c})=\r^{2}T$, where $T$ is the $(0,4)$-tensor given by
\begin{equation*}
T=\Riem_{\Sigma}-\frac{1}{2}g_{\Sigma}\kn g_{\Sigma}.
\end{equation*}
In particular $|\Riem(g_{c})|_{g_{c}}=\r^{2}|T|_{g_{c}}$. Computing in a product chart $((a,b)\times U, \text{id}\times \varphi)$ of $\mathcal{C}_{\Sigma}(a,b)$, we have
\begin{equation*}
|T|_{g_{c}}^{2}=g_{c}^{ai}g_{c}^{bj}g_{c}^{ck}g_{c}^{d\ell}T_{abcd}T_{ijk\ell}=
\r^{-8}g_{\Sigma}^{ai}g_{\Sigma}^{bj}g_{\Sigma}^{ck}g_{\Sigma}^{d\ell}T_{abcd}T_{ijk\ell}=\r^{-8}|T|^{2}_{g_{\Sigma}}
\end{equation*}
and therefore $|\Riem(g_{c})|_{g_{c}}=\r^{-2}|T|_{g_{\Sigma}}$. Being $\Sigma$ compact, there exists a constant $C>0$, depending only on $\Sigma$, such that $|\Riem(g_{c})|_{g_{c}}\leq \r^{-2}C$. The claim follows.
\end{proof}

Inequality \eqref{curvaturecondition} can be read directly on $\Omega$, giving
\begin{equation*}
|\Rm(x)|_g\leq C(\r^{-2}(x)+\varepsilon_{1}^{2}(x)+\varepsilon_{2}(x))
\end{equation*}
for every $x\in \Omega$. For asymptotically conical ends, we then get the following immediate corollary.

\begin{cor}\label{cor.curvatureonend}
Let $E\subseteq \X$ be a $C^{2}$-asymptotically conical end and let $\varphi\colon \mathcal{C}_{\Sigma}(a,\infty)\to E$ be as in Definition \ref{asymptconical}. Then there exists $b>a$ such that 
\begin{equation}\label{curvaturafineconica}
|\Rm(x)|_{g}\leq C \r(x)^{-2}
\end{equation}
for every $x\in F=\varphi(\mathcal{C}_{\Sigma}(b,\infty))$, where $C>0$ is a constant depending only on $n$ and $\Sigma$. Equivalently, if $p\in \X$ is a basepoint, then there exists $r>0$ such that 
\begin{equation}\label{curvaturedecayconicalend}
|\Rm(x)|_g\leq C/d(x,p)^{2}
\end{equation}
for every $x\in E$ with $d(x,p)\geq r$, for some $C=C(n,\Sigma)>0$.
\end{cor}

\begin{proof}
Taking $\varepsilon=1/3$ in \eqref{stimadiffmetrichecono} we have
\begin{equation}
\vert \nabla^{\ell}(\varphi^{\ast}g-g_{c})\vert(x) \leq \frac{1}{3}\r(x)^{-\ell} 
\end{equation}
on $\mathcal{C}_{\Sigma}(b,\infty)$ for some $b>a$, where we can assume $b>a+1$. At this point defining $\varepsilon_{\ell}(x)=\frac{1}{3}\r(x)^{-\ell}$ we have $\varepsilon_{\ell}(x)\leq \frac{1}{3}b^{-\ell}<\frac{1}{2}$ and Proposition \ref{closenessmetric} gives the first claim. Now, from Proposition \ref{propconicalend} we have
\begin{equation*}
\lim_{\r(x)\to \infty}\frac{d(p,x)}{\r(x)}=1.
\end{equation*} 
Therefore there exists $b_{2}>b_{1}$ such that $\frac{1}{2} \r(x)\leq d(x,p)\leq 2\r(x)$ for every $\r(x)>b_{2}$ and the second formula follows immediately.
\end{proof}

Since we are interested in computing the number of conical ends of a Ricci shrinker, it is convenient to introduce the following asymptotically conical version of Definition \ref{def.end1}.

\begin{defn}\label{abstractconicalend} 
An end $\e\in \mathcal{E}(\X)$ as in Definition \ref{def.end1} is \emph{$C^{k}$-asymptotically conical} if there exists a compact subset $K\subseteq \X$ such that the end $\e(K)$ relative to $K$ is $C^{k}$-asymptotically conical in the sense of Definition \ref{asymptconical}. The set of such asymptotically conical ends of $\X$ is denoted by $\mathcal{E}_c(\X)$.
\end{defn}

The relation between this more abstract concept and Definition \ref{asymptconical} is given by the following result.

\begin{prop}\label{fineconica} 
Let $E$ be an end relative to some compact subset $K\subseteq \X$. If $E$ is $C^{k}$-asymptotically conical then there exists a unique end $\e \in \mathcal{E}(\X)$ such that $\e(K)=E$ and $\e$ is $C^{k}$-asymptotically conical.
\end{prop}

\begin{proof}
Take an exhaustion of $\X$ by compact subsets, $(K_{i})_{i\in \N}$, with $K_{0}=K$. Thanks to the second point in Lemma \ref{lemmacono}, for every $i>0$ there exists a unique end relative to $K_{i}$ that is contained in $E$; denote it by $E_{i}$. Therefore the end $\e \in \mathcal{E}(\X)$ defined by $\e(K_{i})=E_{i}$, where $E_{0}=E$, is the unique one satisfying $\e(K)=E$, and we are done.
\end{proof}

We now turn our attention to nearly conical regions as in Definition \ref{nearannulus}. Note that such regions naturally come from asymptotically conical ends. Indeed if $E\subseteq \X$ is a $C^{k}$-asymptotically conical end, we obtain from \eqref{stimadiffmetrichecono} with $\varepsilon \in (0,1/2)$
\begin{equation}
\sup_{x \in \mathcal{C}_{\Sigma}(b,\infty)}\r(x)^{\ell}\vert \nabla^{\ell}(\varphi^{\ast}g-g_{c})\vert(x) \leq \varepsilon 
\end{equation}
for appropriate $b>a$, where $\varphi\colon \mathcal{C}_{\Sigma}(a,\infty)\to E$ is as in Definition \ref{asymptconical}. Therefore for every interval $I\subseteq (b,\infty)$, the subset $\Omega=\varphi(\mathcal{C}_{\Sigma}(I))$ is $C^{k}$-near to the conical annulus $\mathcal{C}_{\Sigma}(I)$.

Also in the case of such nearly conical regions we obtain an estimate for the curvature as follows. Suppose $\Omega$ is $C^{2}$-near to the conical annulus $\mathcal{C}_{\Sigma}(a,b)$.  Applying Proposition \ref{closenessmetric} with $\varepsilon_{\ell}=\varepsilon a^{-\ell}$ we obtain
\begin{equation}\label{curvatureconicalset}
|\Rm|(x)\leq Ca^{-2}
\end{equation}
for every $x\in \Omega$, where $C>0$ depends only on $n$ and $\Sigma$.

We conclude this section with a volume non-collapsing result for nearly conical regions. We will use this to apply Perelman's pseudo-locality theorem (see Proposition \ref{pseudolocality} below). As usual we start with the cone case.

\begin{prop}\label{volumeprep} 
Let $n\in \N$. There exists $c_{0}(n)>0$ with the following property. Let $\Sigma$ be a compact, connected
$(n-1)$-dimensional Riemannian manifold and consider its cone $\mathcal{C}_{\Sigma}$. For every $r>0$ there exists $a_{0}=a_{0}(n,r,\Sigma)>0$ such that, 
if $a\geq a_{0}$ and $b>a+r$, every metric ball $B_{g_{c}}(x_{0},\varrho)$ of $\mathcal{C}_{\Sigma}(a,b)$ with respect to the intrinsic distance, 
satisfies $\Vol_{g_{c}}(B_{g_{c}}(x_{0},\varrho))\geq c_{0}(n)\varrho^{n}$, 
where $\varrho \in (0,r]$. 
\end{prop}

\begin{proof}
Since the metric on the cone $\mathcal{C}_{\Sigma}$ is $g_{c}=d\r \otimes d\r + \r^{2}g_{\Sigma}$, the Riemannian volume form is given by the product measure $\nu_{g_{c}}=(\r^{n-1}\lambda)\otimes \nu_{\Sigma}$, where $\r^{n-1}\lambda$ denotes the measure on $(0,\infty)$ with density function $\r^{n-1}$. Therefore if $A=I\times B$, with $I\subseteq (0,\infty)$ an interval and $B\subseteq \Sigma$ a measurable set, we have $\Vol_{g_{c}}(A)=\left(\int_{I}\r^{n-1}d\r\right)\Vol_{\Sigma}(B)$. Let $r>0$ and $a\geq a_{0}>0$, $b>a+r$, where $a_{0}>0$ will be determined later. For $\varrho \in (0,r]$ consider a metric ball $B_{g_{c}}(x_{0},\varrho)=\{x\in \mathcal{C}_{\Sigma}(a,b)\colon d_{c}(x,x_{0})<\varrho\}$, where $d_{c}$ is the intrinsic distance of $\mathcal{C}_{\Sigma}(a,b)$ and $x_{0}=(\r_{0},y_{0})\in \mathcal{C}_{\Sigma}(a,b)$. The idea is to estimate the volume of this ball using a product set contained in it. We require $a_{0}\geq \max\{r, 2r/\text{inj}(\Sigma)\}$, where
$\text{inj}(\Sigma)>0$ is the injectivity radius of $\Sigma$. In this way, the interval $I$ given by
\begin{equation*}
I=\begin{cases}
(\r_{0}-\varrho/2,\r_{0}), & \text{if } \r_{0}-\varrho/2 > a,  \\
(\r_{0},\r_{0}+\varrho/2), & \text{if } \r_{0}+\varrho/2< b,
\end{cases}
\end{equation*}
is always contained in $(a,b)$ and has length $\varrho/2$. (In the above definition of $I$, if both conditions are satisfied, we can pick $I$ either way.) Let us now prove that $I \times B_{\Sigma}(y_{0},\delta)\subseteq B_{c}(x_{0},\varrho)$, where $\delta=\varrho/2\r_{0}$. Take $x=(\r,y) \in I\times B_{\Sigma}(y_{0},\delta)$. Let $\sigma\colon [0,1]\to \Sigma$ be a minimising geodesic between $y_{0}$ and $y$ and define the curve $\gamma\colon [0,2]\to \mathcal{C}_{\Sigma}(a,b)$ as follows:
\begin{equation*}
\gamma(t)=\begin{cases}
(\r_{0},\sigma(t)), & \text{if } t\in [0,1],\\
(\r_{0}+(t-1)(\r-\r_{0}),y) & \text{if } t\in [1,2].
\end{cases}
\end{equation*}
Then $\gamma$ is a piecewise smooth curve in $\mathcal{C}_{\Sigma}(a,b)$ and we can compute its length:
\begin{equation*}
L_{g_{c}}(\gamma)=\int_{0}^{2}|\dot{\gamma}(t)|dt = \r_{0}\int_{0}^{1}|\dot{\sigma}(t)|dt + \int_{1}^{2}|\r-\r_{0}|dt < \r_{0}\delta+|\r-\r_{0}|<\varrho.
\end{equation*}
Therefore $d_{c}(x_{0},x)\leq L_{g_{c}}(\gamma)<\varrho$ which proves the desired inclusion $I \times B_{\Sigma}(y_{0},\delta)\subseteq B_{g_c}(x_{0},\varrho)$. In particular
\begin{equation*}
\Vol_{g_{c}}(B_{c}(x_{0},\varrho))\geq \Vol_{g_{c}}(I\times B_{\Sigma}(y_{0},\delta))=
\bigg(\int_{I} \r^{n-1}d\r \bigg)\Vol_{\Sigma}(B_{\Sigma}(y_{0},\delta)).
\end{equation*}
Since every $\r \in I$ satisfies $\r>\inf I\geq \r_{0}-\varrho/2$ it follows 
\begin{equation*}
\int_{I}\r^{n-1}d\r \geq  \int_{I}\Big(\r_{0}-\frac{\varrho}{2}\Big)^{n-1}d\r=\frac{\varrho}{2}\Big(\r_{0}-\frac{\varrho}{2}\Big)^{n-1}.
\end{equation*}
On the other hand we have $\delta<\text{inj}(\Sigma)/2$ by the choice of $a_{0}$ and thanks to Proposition 14 in \cite{croke} we can estimate
\begin{equation*}
\Vol_{\Sigma}(B_{\Sigma}(y_{0},\delta))\geq c_{1}(n-1)\delta^{n-1},
\end{equation*}
where the constant $c_{1}(n-1)$ can be explicitly determined. Putting everything together, we get
\begin{equation*}
\Vol_{g_{c}}(B_{g_c}(x_{0},\varrho))\geq c_{1}(n-1)\frac{\varrho}{2}\Big(\r_{0}-\frac{\varrho}{2}\Big)^{n-1}\delta^{n-1}
\geq c_{1}(n-1)2^{1-2n}\varrho^{n}=c_{0}(n)\varrho^{n}
\end{equation*}
which gives the claim.
\end{proof}

\begin{cor}\label{volumeball} 
Let $n\in \N$. There exists $c(n)>0$ with the following property. Consider the cone $\mathcal{C}_{\Sigma}$ over a compact, connected 
$(n-1)$-dimensional Riemannian manifold $\Sigma$. For every $r>0$ there exists $a_{0}=a_{0}(n,r,\Sigma)>0$ such that if $\X$ is an $n$-dimensional singular space and $\Omega\subseteq \reg \subseteq \X$ is an open subset $C^{0}$-near to the conical annulus 
$\mathcal{C}_{\Sigma}(a,b)$, where $a\geq a_{0}$ and $b>a+r$, then $\Vol(B(x_{0},\varrho))\geq c(n)\varrho^{n}$ for every ball 
$B(x_{0},\varrho)\subseteq \Omega$ with $\varrho \in (0,r]$. 
\end{cor}

\begin{proof}
Let $\varphi\colon \mathcal{C}_{\Sigma}(a,b)\to \Omega$ be the diffeomorphism given by Definition \ref{nearannulus} and let $\psi$  be its inverse. Set $h=\varphi^{\ast}g\vert_{\Omega}$. Since the map $\psi\colon (\Omega,g\vert_{\Omega})\to (\mathcal{C}_{\Sigma}(a,b),h)$ is an isometry and $B_{g}(x_{0},\varrho)\subseteq \Omega$, we have $\psi(B_{g}(x_{0},\varrho))=B_{h}(\overline{x}_{0},\varrho)$ where $\overline{x}_{0}=\psi(x_{0})$ and $B_{h}(\overline{x}_{0},\varrho)=\{x\in \mathcal{C}_{\Sigma}(a,b)\colon d_{h}(x,\overline{x}_{0})<\varrho\}$. Moreover $\Vol_{g}(B_{g}(x_{0},\varrho))=\Vol_{h}(B_{h}(\overline{x}_{0},\varrho))$. Since $(1-\varepsilon)g_{c}\leq h \leq (1+\varepsilon)g_{c}$ on $\mathcal{C}_{\Sigma}(a,b)$ it follows (using the notation of Proposition \ref{volumeprep}) that $B_{h}(\overline{x}_{0},\varrho)\supseteq B_{g_{c}}(\overline{x}_{0},\varrho/\sqrt{1+\varepsilon})$ and $\frac{d\nu_{h}}{d\nu_{g_{c}}}\geq (1-\varepsilon)^{n/2}$. If $a\geq a_{0}$, with $a_{0}$ given by the preceding Proposition \ref{volumeprep}, we have
\begin{align*}
\Vol_{g}(B_{g}(x_{0},\varrho))&=\Vol_{h}(B_{h}(\overline{x}_{0},\varrho))\geq (1-\varepsilon)^{n/2}\Vol_{g_{c}}(B_{h}(\overline{x}_{0},\varrho))\\
&\geq (1-\varepsilon)^{n/2}\Vol_{g_{c}}\big(B_{g_{c}}(\overline{x}_{0},\varrho/\sqrt{1+\varepsilon})\big)
\geq \Big(\frac{1-\varepsilon}{1+\varepsilon} \Big)^{n/2} \! c_{0}(n)\varrho^{n}\\
& \geq 3^{-n/2}c_{0}(n)\varrho^{n}=c(n)\varrho^{n}
\end{align*}
and we are done.
\end{proof}

\begin{rem}\label{volumeball-rem} 
Observe that in Proposition \ref{volumeprep} and Corollary \ref{volumeball} 
one can explicitly take $a_{0}=a_{0}(n,r,\Sigma)= \max\{r,2r/\text{inj}(\Sigma)\}$.
\end{rem}

It is worth pointing out that, in the case of smooth Ricci shrinkers, from the curvature estimate \eqref{curvatureconicalset} it is possible to obtain a result similar to Corollary \ref{volumeball} using the following theorem.

\begin{prop}[Li-Wang \cite{heatkernel}]\label{liwangteo} 
Let $M$ be a smooth Ricci shrinker with entropy $\mu(g)$ and suppose $S\leq A$ on the ball $B(x,\varrho)\subseteq M$. Then
\begin{equation*}
\Vol(B(x,\varrho))\geq c(n)e^{\mu(g)-A \varrho^{2}}\varrho^{n}
\end{equation*}
where $c(n)$ is a constant depending only on $n$.
\end{prop}

\section{Proof of Theorem \ref{thm.conicalannulus}}\label{nearlyconical-section}

We want to prove that if a Ricci shrinker contains a region that is near to a cone then the shrinker is asymptotically conical along an end containing that region. The main ingredients of the proof are the volume estimate from Corollary \ref{volumeball} and the following generalisation to Ricci shrinkers of Perelman's pseudo-locality theorem \cite{perelman}, due to Li and Wang \cite{heatkernel2}.

\begin{prop}[Li and Wang \cite{heatkernel2}]\label{pseudolocality} 
Let $n\in \N$ and $v\in (0,1)$. There exists a positive constant  $\varepsilon_{0}=\varepsilon_{0}(n,v)>0$ with the following property. 
Fix $r>0$ and let $(g_{t})_{t<1}$ be the Ricci flow associated with an $n$-dimensional Ricci shrinker $(M,g,f)$. 
Suppose that, for a given $(x_{0},t_{0})\in M\times (-\infty,1)$, the following curvature and volume conditions hold:
\begin{enumerate}
\item $|\Rm(y,t_{0})|\leq (vr)^{-2}$ for every $y\in B_{t_{0}}(x_{0},r)$;
\item $\Vol_{t_{0}}(B_{t_{0}}(x_{0},r))\geq vr^{n}$.
\end{enumerate} 
Then 
\begin{equation*}
|\Rm(y,t)|\leq (\varepsilon_{0}r)^{-2}
\end{equation*}
for every $y\in B_{t_{0}}(x_{0},(1-v)r)$ and every 
$t\in [t_{0}-(\varepsilon_{0}r)^{2},t_{0}+(\varepsilon_{0}r)^{2}]\cap (-\infty,1)$.
\end{prop}

\begin{rem} The main differences between Proposition \ref{pseudolocality} and the original Perelman result (Theorem 10.3 in \cite{perelman}) are the volume condition
$\Vol_{t_{0}}(B_{t_{0}}(x_{0},r))\geq v r^{n}$, which replaces the stronger requirement 
$\Vol_{t_{0}}(B_{t_{0}}(x_{0},r))\geq (1-\varepsilon_{0})\omega_{n}r^{n}$, where $\omega_{n}$ denotes the volume of the unit ball of 
$\R^{n}$, and the absence of any curvature assumption on $M$. 
Other versions of the pseudo-locality theorem for general Ricci flows 
require instead a global bound on the curvature; see for example \cite{peng} and \cite{simontopping}. 
\end{rem}

We are ready to prove Theorem \ref{thm.conicalannulus} from the introduction.

\begin{proof}[Proof of Theorem \ref{thm.conicalannulus}.]
Before we start, observe that the claim is trivially true if $M$ is the Gaussian shrinker, and therefore we can suppose $M$ not flat, or in other terms $S>0$. In this case $\varrho$ is everywhere smooth and strictly positive. The main idea of the proof is to find $a_{\ast}$ such that Proposition \ref{pseudolocality} can be applied to large balls contained in $\Omega$. In the following we always suppose that $a,b$ and $\alpha,\beta$ satisfy the desired conditions. 

Since $\Omega$ is $C^{2}$-near to $\mathcal{C}_{\Sigma}(a,b)$, thanks to \eqref{curvatureconicalset} we have $|\Rm|\leq C/a^{2}$ on $\Omega$, with $C>0$ depending only on $n$ and $\Sigma$. Using this $C$ and setting $r_{0}=\varepsilon_{0}^{-1}$, where 
$\varepsilon_{0}$ is given by Proposition \ref{pseudolocality} with $v=\min\{1/2,c(n)\}$, 
if $a\geq r_{0}\sqrt{C+1}$ it follows $|\Rm|\leq r_{0}^{-2}$ on $\Omega$. In particular, $S\leq n^2r_{0}^{-2}$ on $\Omega$ and recalling the fundamental relation \eqref{auxilliaryconstant} we have
\begin{equation*}
\vert \nabla f \vert^{2}=\varrho^{2}/4-S \geq \alpha^{2}/4-n^2r_{0}^{-2}=\frac{1}{4} (\alpha^{2}-4n^2r_{0}^{-2})
\end{equation*}
on $V$. If we take $\alpha \geq r_{0}\sqrt{C+1}+2nr_{0}^{-1}$ it follows $\vert \nabla f \vert >0$ on $V$. Therefore we require $a_{\ast}\geq r_{0}\sqrt{C+1}+2nr_{0}^{-1}$. With this choice, $\varrho(V)=(\alpha,\beta)$.

Now let $\xi \in (\alpha+r_{0},\beta-r_{0})$ be a regular value for $f$, where this interval is not empty being $\alpha>r_{0}$ and 
$\beta>3\alpha$. If $L=D(\xi)\cap V$ then $B(x_{0},r_{0})\subseteq V$ for every $x_{0}\in L$. Indeed, since $\varrho$ is $1$-Lipschitz, we have $|\varrho(x)-\varrho(x_{0})|\leq r_{0}$ for every $x$ in the ball and therefore
\begin{align*}
\varrho(x) &\leq \varrho(x_{0})+r_{0}=\xi+r_{0}<\beta,\\
\varrho(x) &\geq \varrho(x_{0})-r_{0}=\xi-r_{0}>\alpha.
\end{align*}
By connectedness $B(x_{0},r_{0})\subseteq V$ for every $x_{0}\in L$. Now let $a_{0}=a_{0}(\Sigma,n,r_{0})=a_{0}(\Sigma,n)$ be the constant given by Corollary \ref{volumeball} and define  $a_{\ast}=\max\{r_{0}\sqrt{C+1}+2nr_{0}^{-1}, a_{0}\}$. Corollary \ref{volumeball} yields $\Vol(B(x_{0},r_{0}))\geq c(n)r_{0}^{n}\geq v r_{0}^{n}$ for every ball with $x_{0}\in L$. At this point we can apply the pseudo-locality theorem, Proposition \ref{pseudolocality}, with $t_{0}=0$, obtaining 
\begin{equation*}
|\Rm(y,t)|\leq \varepsilon_{0}r_{0}=1
\end{equation*}
for every $y\in B(x_{0},(1-v)r_{0})$ and $t\in [-1,1)$, where $x_{0}\in L$. 
Using the self-similarity of $g_{t}$, the preceding condition can be restated as
\begin{equation}
|\Rm(y)|\leq 1-t
\end{equation}
for every $y\in \psi_{t}(B(x_{0},(1-v)r_{0}))$ and $t\in [-1,1)$. 
In particular $|\Rm(\psi_{t}(x_{0}))|\leq 1-t$ for every $x_{0}\in L$ and $t\in [0,1)$.

Now take some $x\in L$. From (\ref{evolution2}) and $|\nabla f|^{2} \leq f+c(g)$, we obtain the differential inequality
\begin{equation*}
\frac{d}{dt}\left(f\circ \psi_{t}(x)\right)\leq \frac{f(\psi_{t}(x))+c(g)}{1-t}.
\end{equation*}
Considering the (positive) function $f(\psi_{t}(x))+c(g)$ and integrating, it follows that
\begin{equation*}
f(\psi_{t}(x))+c(g)\leq \frac{f(x)+c(g)}{1-t}
\end{equation*}
for every $t\in [0,1)$. Therefore
\begin{equation}\label{disugriem}
|\Rm(\psi_{t}(x))|\leq 1-t \leq \frac{f(x)+c(g)}{f(\psi_{t}(x))+c(g)}=\frac{\varrho^{2}(x)}{\varrho^{2}(\psi_{t}(x))}
\end{equation}
for every $t\in [0,1)$.

Let $W$ be the connected component of $D[\xi,\infty)$ containing $L$. Note that $L$ is connected since $V$ does not contain any critical point of $f$. We claim that the flow of $\nabla f$ gives a diffeomorphism $\phi \colon L\times [0,\infty)\to W$. Let us first verify that $\phi$ is injective. Indeed $\phi(x_{0},t_{0})=\phi(x_{1},t_{1})$ if and only if $\phi(x_{1},t_{1}-t_{0})=x_{0}$. But since $\varrho$ is non-decreasing along the flow lines and there are no critical points on $L$, this implies $t_1-t_0=0$ and $x_1=x_0$. Now, let us prove that $\phi$ is onto. Take $y \in W$. If $y \in L$ we are done, so suppose $y\notin L$ and consider a continuous curve $\gamma\colon [0,1]\to W$ from a point $q\in L$ to $y$. Define
\begin{equation*}
I=\{t\in [0,1]\colon \gamma(t)=\phi(x,s)\ \text{for some}\ (x,s)\in L\times [0,+\infty)\}.
\end{equation*}
If we prove that $I$ is open and closed in $[0,1]$ then $I=[0,1]$ since $0\in I$, and the claim follows. As $V\subseteq \phi(L\times \R)$ and $\varrho(\gamma(t))\geq \xi$, it is easy to see that $I=\gamma^{-1}\big( \bigcup_{s\geq 0}\phi_{s}(V)\big)$ and thus $I$ is open in $[0,1]$. To show that $I$ is closed, let $(t_{i})_{i \in \N}\subseteq I$ be a sequence converging to $t_{\ast}\in [0,1]$. 
By assumption, $\gamma(t_{i})=\phi(x_{i},s_{i})$ for some $x_{i}\in L$ and $s_{i}\geq 0$. Note that $(s_{i})_{i \in \N}$ is bounded. Otherwise, assuming $s_{i}\to \infty$, \eqref{disugriem} and \eqref{cambioparametro2} imply
\begin{equation*}
\vert \Rm(\phi(x_{i},s_{i}))\vert =\vert \Rm(\psi(x_{i},1-e^{-s_{i}}))\vert \leq 1-(1-e^{-s_{i}})=e^{-s_{i}}\to 0
\end{equation*}
as $i\to \infty$ and therefore $\vert \Rm(\gamma(t_{i}))\vert \to 0$; but $\gamma(t_{i})\to \gamma(t_{\ast})$, which implies $\vert\Rm(\gamma(t_{\ast}))\vert=0$, contradicting the assumptions on $M$. It follows, up to a subsequence, that $s_{i}\to s_{\ast}\in [0,\infty)$. Since $L$ is compact we can suppose $x_{i}\to x_{\ast}\in L$. Therefore $\phi(x_{i},s_{i})$ converges to $\phi(x_{\ast},s_{\ast})$ and $\gamma(t_{\ast})=\phi(x_{\ast},s_{\ast})$. This gives $t_{\ast}\in I$ proving that $I$ is indeed closed. In conclusion, $\phi\colon L\times [0,\infty)\to W$ is a diffeomorphism and in particular $W$ contains no critical points of $f$.

From Proposition \ref{prop.ends-shrinker} we know that $W$ is a connected submanifold with boundary of $M$, where $\partial{W}=L$. Moreover $\varrho(W)=[\xi,\infty)$ and $W^{\circ}$ is an end of $M$ relative to $D[0,\xi]$. Since $W$ does not contain any critical point of $f$, again from Proposition \ref{prop.ends-shrinker} there exists a unique end $\e\in \mathcal{E}(M)$ such that $\e(D[0,\xi])=W^{\circ}$. In particular $\e(D[\alpha,\infty))$ contains $V$, since $W\subseteq \e(D[\alpha,\infty))$ and $W\cap V\neq \emptyset$. Denote $E=\e(D[0,\alpha])$ for short.

Consider $V$ and $W$. First of all $V\cap W \neq \emptyset$. Moreover $V\cap D[\xi,\infty)\subseteq W$, being connected and intersecting $W$; similarly $W\cap D(\alpha,\beta)=W\cap D[\xi,\beta)\subseteq V$ since it is connected and intersects $V$. Now the set $A=V\cup W=V\cup W^{\circ}$ is obviously connected and is open. Since $A\subseteq D(\alpha,\infty)$ and it intersects $E$, we conclude $A\subseteq E$. We claim that indeed $A=E$. In order to prove this, take points $x\in A$, $y\in E$ and a curve $\gamma\colon [0,1]\to E$ connecting them. Consider $I=\{t\in [0,1]\colon \gamma(t)\in A\}$. Then $I\neq \emptyset$ because $0\in I$ and it is open in $[0,1]$. To show that $I$ is closed, take a sequence $(t_{i})_{i \in \N}$ in $I$ converging to $t_{\ast}\in [0,1]$. Then $x_{i}=\gamma(t_{i}) \in A$ and $x_{i}\to x_{\ast}=\gamma(t_{\ast})\in E$ by continuity. Let $\lambda_{\ast}=\varrho(x_{\ast})$ and take $\delta>0$ to be determined. Then there exists $\varepsilon>0$ such that $\varrho(\gamma(t))\in (\lambda_{\ast}-\delta,\lambda_{\ast}+\delta)$ for every $t \in I_{\varepsilon}=(t_{\ast}-\varepsilon,t_{\ast}+\varepsilon)$. Fix a time $t_{i}\in I_{\varepsilon}$ and consider $x_{i}\in A$. Suppose $\lambda_{\ast}>\xi$. If $\delta<\lambda_{\ast}-\xi$ then $\varrho(\gamma(t))>\xi$ for $t\in I_{\varepsilon}$ and therefore $x_{i}\in A\cap D(\xi,\infty)$. In particular $x_{i}\in W$ and the restriction of $\gamma$ between $t_{i}$ and $t_{\ast}$ is a curve in $D(\xi,\infty)$ starting from $x_{i}\in W$. It follows that $x_{\ast}\in W$. Now, suppose $\varrho(x_{\ast})\leq \xi$. If $\delta<\min\{\lambda_{\ast}-\alpha,\beta-\lambda_{\ast}\}$ then $\varrho(\gamma(t))\in (\alpha,\beta)$ for $t\in I_{\varepsilon}$ and $x_{i}\in A\cap D(\alpha,\beta)=V$. Therefore the restriction of $\gamma$ between $t_{i}$ and $t_{\ast}$ is a curve in $D(\alpha,\beta)$ starting from $x_{i}\in V$. This implies $x_{\ast}\in V$ and we are done, having $x_{\ast}\in A$ in both cases. This proves that $I$ is closed and therefore $I=[0,1]$ giving $y\in A$ and thus $E=A$.

Thanks to the equality, $E$ does not contain any critical point of $f$. In particular $\e$ is the unique end of $M$ with $\e(D[0,\alpha])=E$ and furthermore the only one with $V\subseteq \e(D[0,\alpha])$. Indeed if $\widetilde{\e}\in \mathcal{E}(M)$ satisfies $V\subseteq \widetilde{\e}(D[0,\alpha])$ then necessarily $\e(D[0,\alpha])=\widetilde{\e}(D[0,\alpha])$ and therefore $\e=\widetilde{\e}$.

It only remains to prove that $\e$ is asymptotically conical. Using \eqref{cambioparametro}, it is easy to see that the map
\begin{equation*}
\psi\colon \partial W \times [0,1) \to W,\ (t,x)\mapsto \psi_{t}(x)
\end{equation*}
is a diffeomorphism. At this point take $y\in W$. Then $y=\psi_{t}(x)$ for some $(x,t)\in \partial W \times [0,1)$ and using \eqref{disugriem} we have
\begin{equation*}
|\Rm(y)|=|\Rm(\psi_{t}(x))|\leq \frac{\varrho^{2}(x)}{\varrho^{2}(\psi_{t}(x))}=\frac{\varrho^{2}(x)}{\varrho^{2}(y)}=\frac{\xi^{2}}{\varrho^{2}(y)}.
\end{equation*}
This proves that $\Rm$ is quadratically decaying to zero as $|y|\to \infty$ along the end $W^{\circ}$, which implies that $W^{\circ}$ is smoothly asymptotically conical thanks to the argument in Section 2.2 of \cite{kw}, finishing the proof.
\end{proof}

\section{Estimating the number of asymptotically conical ends}\label{conicalends-estimate-section}

In this section, we prove Theorems \ref{thm.conicalbound} and \ref{thm.limitconicalbound}. We have seen in Corollary \ref{cor.curvatureonend} that if $E$ is an asymptotically conical end, then the curvature decays quadratically along $E$,
\begin{equation*}
|\Rm(x)|\leq Cd(p,x)^{-2},
\end{equation*}
where $C$ is a positive constant depending on $E$. In the smooth case, the condition is indeed equivalent to $E$ being asymptotically conical (see e.g. \cite{kw}).

This curvature decay allows us to perform a counting argument for conical ends similar to the one in Section \ref{section-estimate-number-ends}. We start with a result similar to Proposition \ref{endspecialcase}.

\begin{prop}\label{splitconicalend} 
Let $\X$ be a singular Ricci shrinker and let $\e_{1},\ldots,\e_{k}$ be distinct asymptotically conical ends of $\X$. Then for every $r>0$ there exists a compact subset $K$ containing $D[0,r]$ such that $E_{1}=\e_{1}(K),\ldots,E_{k}=\e_{k}(K)$ are distinct asymptotically conical ends relative to $K$.
\end{prop}

\begin{proof}
Let $r>0$ be given and consider an asymptotically conical end $\e_{i}$. By Definition \ref{abstractconicalend}, we can find a compact subset $K_{i}$ such that $F_{i}=\e_{i}(K_{i})$ is an asymptotically conical end relative to $K_{i}$, and therefore there is a diffeomorphism $\varphi_{i}\colon \mathcal{C}_{\Sigma_{i}}(a_{i},\infty)\to F_{i}$ as in Definition \ref{asymptconical}. If we take $s>r$ such that $\bigcup_{i=1}^{k}K_{i}\subseteq D[0,s]$, it is easy to verify that the ends $\e_{1}(D[0,s]), \ldots, \e_{k}(D[0,s])$ are distinct. Now using \eqref{radialdistancebasepoint} we can find $b_{i}>a_{i}$ such that $E_{i}=\varphi_{i}(\mathcal{C}_{\Sigma_{i}}(b_{i},\infty))$ is contained in $D(s,\infty)$. Moreover, Points 3 and 4 of Proposition \ref{propconicalend} imply that $E_{i}$ is the unique end relative to the compact subset $H_{i}=K_{i}\cup \varphi_{i}(\mathcal{C}_{\Sigma_{i}}(a_{i},b_{i}])$ that is contained in $F_{i}$. Furthermore, this end coincides with $\e_{i}(H_{i})$ and is asymptotically conical. By construction $E_{1},\ldots,E_{k}$ are pairwise disjoint and an easy argument proves that they are ends relative to the compact subset given by $K=D[0,s]\cup \bigcup_{i=1}^{k}H_{i}$.
\end{proof}

In order to estimate the number of conical ends, some care is needed. Indeed the constant $C$ appearing in \eqref{curvaturedecayconicalend} depends on $E$, more precisely on the Riemannian manifold $(\Sigma,g_{\Sigma})$ to whose cone $E$ is asymptotic, and therefore it is not possible, at least in general, to get uniform estimates. For this reason we restrict the number of these possible $\Sigma$. We first prove the following version of Theorem \ref{thm.conicalbound}, which also assumes an entropy bound.

\begin{prop}\label{estimate-conical-ends}
Given $n \in \mathbb{N}$, $\underline{\mu}>-\infty$, and a finite collection $\h$ of connected compact Riemannian manifolds of dimension $n-1$, there exists $\Lambda=\Lambda(n,\underline{\mu},\h)$ with the following property. Let $M$ be a smooth $n$-dimensional Ricci shrinker with entropy bounded below $\mu\geq \underline{\mu}$ and suppose that every asymptotically conical end of $M$ has asymptotic cone $\mathcal{C}_{\Sigma}$ for some $\Sigma\in \h$. Then $M$ has at most $\Lambda$ asymptotically conical ends.
\end{prop}

\begin{proof}
Let $\e_{1},\ldots,\e_{k}$ be distinct conical ends of $M$. By the preceding Proposition \ref{splitconicalend}, we can find a compact subset $K\subseteq \X$ such that $E_{1}=\e_{1}(K),\ldots, E_{k}=\e_{k}(K)$ are distinct conical ends relative to $K$. Thanks to Proposition \ref{curvaturafineconica} applied to each of these ends, there exists $r_{0}>5n$ such that $K\subseteq D[0,r_{0})$ and  
\begin{equation*}
|\Rm(x)|\leq C_{i}d(x,p)^{-2}
\end{equation*}
for every $x\in E_{i}$ with $d(x,p)\geq r_{0}$, for $i=1,\ldots,k$.  As seen in Corollary \ref{cor.curvatureonend}, the constants $C_{i}$ depend only on $n$ and $(\Sigma_{i},g_{\Sigma_{i}})\in \h$. Therefore, as $\h$ is finite, we can write $C_{i}\leq C_{\h}$ for some constant $C_{\h}>0$. At this point we can use a standard counting argument.

Take $\xi=3r_{0}$ and choose a point $x_{i}\in E_{i}\cap D(\xi)$ for every $i=1,\ldots,k$. A standard computation shows that each ball $B(x_{i},r_{0})$ is contained in $E_{i}\cap D(2r_{0},4r_{0})$ and therefore they are pairwise disjoint. Using \eqref{rhogrowth} as usual it follows $B(x_{i},r_{0})\subseteq B(p,r)$, where $r=4r_{0}+5n$. Now we can estimate the scalar curvature, getting
\begin{equation*}
S(y)\leq \frac{n^2C_{\h}}{d(y,p)^{2}}\leq \frac{n^2C_{\h}}{(2r_{0}-\sqrt{2n})^{2}}\leq \frac{n^2C_{\h}}{r_{0}^{2}}
\end{equation*} 
for every $y \in B(x_{i},r_{0})$. Therefore $S(y)r_{0}^{2}\leq n^2C_{\h}$ and Proposition \ref{liwangteo} gives
\begin{equation}\label{stimavolume}
\Vol(B(x_{i},r_{0}))\geq c(n) e^{\underline{\mu} -n^2C_{\h}}r_{0}^{n}.
\end{equation}
Using the volume growth \eqref{volumegrowth}, we deduce that the number of disjoint balls of radius $r_{0}$ contained in $B(p,r)$ and satisfying \eqref{stimavolume} is at most 
\begin{equation*}
\Lambda=\Lambda(n,\underline{\mu},\h)=\frac{C_{0}(n)5^{n}}{c(n)e^{\underline{\mu}-n^2C_{\h}}},
\end{equation*}
and we are done. 
\end{proof}

\begin{rem}
Clearly, instead of assuming the condition on finitely many different types of asymptotically conical ends, the result also holds if we start with the stronger \emph{global} curvature decay assumption $|\Rm(x)|\leq C d(x,p)^{-2}$.
\end{rem}

If we use Corollary \ref{volumeball}, we can remove the entropy assumption from the preceding result. 

\begin{proof}[Proof of Theorem \ref{thm.conicalbound}.]
Assume first that we have a smooth Ricci shrinker $(M,g,f)$. Let $\e_{1},\ldots,\e_{k}$ be distinct asymptotically conical ends of $M$. By Proposition \ref{splitconicalend}, we can find a compact subset $K\subseteq M$ and diffeomorphisms $\varphi_{i}\colon \mathcal{C}_{\Sigma_{i}}(a_{i},\infty)\to E_{i}$ as in Definition \ref{asymptconical}, such that $E_{1}=\e_{1}(K),\ldots, E_{k}=\e_{k}(K)$ are distinct conical ends relative to $K$. Moreover, looking at the proof, it is easy to see that we can indeed take $a_{i}=a$ for every $i$ and suppose 
\begin{equation}\label{metric-estim}
|g_{c}^{i}-\varphi_{i}^{\ast}g|\leq \varepsilon
\end{equation}
on $\mathcal{C}_{\Sigma_{i}}(a,\infty)$, where $g_{c}^{i}$ is the conical metric on $\mathcal{C}_{\Sigma_{i}}$ and $\varepsilon \in (0,1/2)$. 

Denote by $\r_{i}\colon E_{i}\to (a,\infty)$ the radial coordinate on $E_{i}$. Thanks to Proposition \ref{propconicalend} together with \eqref{rhogrowth} there exists $r_{0}>5n$ such that $K\subseteq D[0,r_{0})$ and 
\begin{equation}\label{dist-estim}
\frac{1}{2}\r_{i}(x)\leq \varrho(x)\leq 2\r_{i}(x)
\end{equation} 
for every $x\in E_{i}$ with $\varrho(x)>r_{0}$. Considering Corollary \ref{volumeball} and Remark \ref{volumeball-rem}, we define $\sigma=\min\{\text{inj}(\Sigma)\colon \Sigma \in \h\}$ and $\alpha=2r_{0}(1+\sigma^{-1})$. Letting $\Omega_{i}=\varphi_{i}^{-1}(\mathcal{C}_{\Sigma_{i}}(\alpha,\infty))\subseteq E_{i}$, \eqref{metric-estim} and \eqref{dist-estim} imply that $\Omega$ is $C^{0}$-near to the cone $\mathcal{C}_{\Sigma_{i}}(\alpha,\infty)$ and is contained in $E_{i}\cap D(r_{0},\infty)$. For every $i$, choose a point $x_{i}\in E_{i}\cap D(\xi)$, where $\xi=2\alpha+r_{0}$, and consider the ball $B(x_{i},r_{0})$. A standard computation gives $B(x_{i},r_{0})\subseteq E_{i}\cap D(2\alpha,\xi+r_{0})$ and in particular these balls are pairwise disjoint. Moreover, any $y \in B(x_{i},r_{0})$ satisfies $\r_{i}(y)>\frac{1}{2}\varrho(y)>\alpha$ due to \eqref{dist-estim}, and therefore $B(x_{i},r_{0})\subseteq \Omega_{i}$, being $\varphi_{i}$ one to one. At this point we can apply Corollary \ref{volumeball}, which gives 
\begin{equation}\label{vol-estim}
\Vol(B(x_{i},r_{0}))\geq c(n)r_{0}^{n}.
\end{equation}
On the other hand, $B(x_{i},r_{0})\subseteq B(p,r)$, where $r=2\alpha+4r_{0}$. Using the volume growth \eqref{volumegrowth}, it follows that the number of disjoint balls satisfying \eqref{vol-estim} and contained in 
$B(p,r)$ is at most
\begin{equation*}
\Lambda=\Lambda(n,\h)=\frac{C_{0}(n)4^{n}(2+\sigma^{-1})^{n}}{c(n)}
\end{equation*}
and the claim follows. 

Finally, note that the preceding argument can be extended, with minor modifications, to the case of limit singular Ricci shrinkers, using in particular the volume growth estimates from Proposition \ref{prop.limitsingularRS} in the last part of the argument.
\end{proof}

Finally, using Theorem \ref{thm.conicalannulus}, proved in the previous section, we can show that no new conical end appears in a limit of a sequence of Ricci shrinkers.

\begin{proof}[Proof of Theorem \ref{thm.limitconicalbound}.]
The proof is based on the same idea already seen in the proof of Theorem \ref{thm.limitendbound}, that is constructing an injective map $\lambda_{i}\colon \mathcal{E}_{c}(\X)\to \mathcal{E}_{c}(M_{i})$ for every $i$ large enough. As usual $(U_{i})_{i\in \N}$ and $(\phi_{i}\colon U_{i}\to M_{i})_{i \in \N}$ are as in Definition \ref{def.singularCG}.

First, we can suppose $\X$ is not compact and with at least one conical end, otherwise the claim is trivial. Note that for the moment this does not imply that $M_{i}$ is not compact for $i$ large!

Suppose the set $\mathcal{E}_{c}(\X)$ is finite. By definition every $\e\in \mathcal{E}_{c}(\X)$ is asymptotic to a cone over some manifold $\Sigma_{\e}$ and we let $a_{\ast}(\e)=a_{\ast}(n,\Sigma_{\e})$ be the positive constant given by Theorem \ref{thm.conicalannulus}. We also denote $a_{\ast}=\max\{a_{\ast}(\e)\colon \e\in \mathcal{E}_{c}(\X)\}+1$.

Thanks to Proposition \ref{splitconicalend}, we can find a compact subset $K\subseteq \X$ containing $D[0,a_{\ast}]$ such that 
$E_{\e}=\e(K)$, for $\e \in \mathcal{E}_{c}(\X)$, are distinct conical ends relative to $K$. In particular for every $\e$ we get a diffeomorphism $\varphi_{\e}\colon \mathcal{C}_{\Sigma_{\e}}(a_{\e},\infty)\to E_{\e}$ as in Definition \ref{asymptconical}. If we look at the proof of Proposition \ref{splitconicalend}, it is easy to see that, for every $\e$, we can suppose $a_{\e}=a_{\diamond} >a_{\ast}$ and
\begin{equation}\label{closenesslimitend}
\sup_{\mathcal{C}_{\Sigma_{\e}}(a_{\diamond},\infty)} \r(x)^{\ell}|\nabla^{\ell}(\varphi_{\e}^{\ast}g-g_{c})|\leq \varepsilon
\end{equation} 
for every $0\leq \ell \leq 2$, with $\varepsilon \in (0,1/10)$. It follows that $|\Rm|\leq C/a_{\diamond}^{2}$ on $E_{\e}$ and, arguing as in the proof of Theorem \ref{thm.conicalannulus}, we get $|\nabla f|>0$ on $E_{\e}$, by the definition of $a_{\ast}(\e)$. Note that $E_{\e}\subseteq \reg$ and so we can use \eqref{auxilliaryconstant} without problems.

Let $s>a_{\ast}$ be such that $K\subseteq D[0,s]$ and consider an end $\e \in \mathcal{E}_{c}(\X)$. In the following we drop the subscript $\e$. From Proposition \ref{propconicalend} we know that there exists only one end relative to $D[0,s]$ that is contained in $E$, and we denote it by $E(s,\infty)$ as usual. Similarly, applying Proposition \ref{prop.ends-shrinker} to $E(s,\infty)$, if $I\subseteq (s,\infty)$ is an interval then there exists a unique connected component of $D(I)$ that is contained in $E(s,\infty)$; denote it by $E(I)$. Moreover $\psi(E(I))\subseteq \mathcal{C}_{\Sigma}(a,\infty)$, where $\psi=\varphi^{-1}$, 
$\varrho(E(I))=I$ and $E(I)$ is compact if $I$ is. Now choose any $\alpha_{0},\beta_{0}>s$ with $\beta_{0}>3\alpha_{0}+1$ and consider $E[\alpha_{0},\beta_{0}]$. Being compact there exist $a>a_{\diamond}$ and $b>3a$ such that $\psi(E[\alpha_{0},\beta_{0}])\subseteq \mathcal{C}_{\Sigma}[a,b]$. If we define $\Omega=\varphi(\mathcal{C}_{\Sigma}(a,b))$ then $\Omega \Subset \reg$ and therefore there exists $i_{0}\in \N$ such that $\Omega\subseteq U_{i}$ for every $i\geq i_{0}$. 

\begin{claim}\label{primo} There exist an interval $(\alpha,\beta)\subseteq (s,\infty)$, with $\beta>3\alpha$, and an integer $i_{\ast}\geq i_{0}$ such that for every $i\geq i_{\ast}$ a connected component $V_{i}$ of $D_{i}(\alpha,\beta)$ is contained in $\phi_{i}(E[\alpha_{0},\beta_{0}])$.
\begin{proof}
The proof is almost identical to the one of the analogue Claim \ref{cc} seen in Theorem \ref{thm.limitendbound}, so we omit it. 
Observe only that the constants $\alpha, \beta$ one finds satisfy $\beta>3\alpha$, as desired.
\end{proof}
\end{claim}

Now define $\Omega_{i}=\phi_{i}(\Omega)\subseteq M_{i}$ for every $i\geq i_{\ast}$.

\begin{claim}\label{secondo}  There exists $i_{\ast\ast}\geq i_{\ast}$ such that every $\Omega_{i}$ with $i\geq i_{\ast\ast}$ is $C^{2}$-near to the conical annulus $\mathcal{C}_{\Sigma}(a,b)$ and contains a connected component $V_{i}$ of $D_{i}(\alpha,\beta)$. 
\begin{proof}
Since the inclusion $V_{i}\subseteq \Omega_{i}$ follows from Claim \ref{primo} it is sufficient to prove the closeness condition. As we are working with various metrics, we will write them explicitly in order to avoid any confusion. In the following, $C$ will denote a positive constant depending on $n$ and $\ell=0,1,2$ and, as always, it can vary from line to line. Let $\delta \in (0,\frac{1}{2})$ to be determined later. Because the convergence is smooth on $\reg$, there exists $i_{\ast\ast}\geq i_{\ast}$ such that for every $i\geq i_{\ast\ast}$ we have 
\begin{equation}\label{stimacondelta}
|\nabla^{\ell}_{g}(\phi_{i}^{\ast}g_{i}-g)|_{g}\leq \delta
\end{equation}
on $\Omega$. Define the diffeomorphisms 
\begin{equation*}
\varphi_{i}=\phi_{i}\vert_{\Omega}\circ \varphi\vert_{\mathcal{C}_{\Sigma}(a,b)}\colon \mathcal{C}_{\Sigma}(a,b)\to \Omega_{i}.
\end{equation*}
Working on the conical annulus $\mathcal{C}_{\Sigma}(a,b)$ we get
\begin{align*}
|\nabla^{\ell}_{g_{c}}(\varphi_{i}^{\ast}g_{i}-g_{c})|_{g_{c}} & \leq |\nabla^{\ell}_{g_{c}}(\varphi_{i}^{\ast}g_{i}-\varphi^{\ast}g)|_{g_{c}} + |\nabla^{\ell}_{g_{c}}(\varphi^{\ast}g-g_{c})|_{g_{c}} \\
& \leq  |\nabla^{\ell}_{g_{c}}(\varphi_{i}^{\ast}g_{i}-\varphi^{\ast}g)|_{g_{c}} + \varepsilon/a^{\ell},
\end{align*}
where we used \eqref{closenesslimitend} in the last inequality. Untangling the definition of the maps $\phi_{i}$, we have
\begin{align*}
\nabla^{\ell}_{g_{c}}(\varphi_{i}^{\ast}g_{i}-\varphi^{\ast}g) = \nabla^{\ell}_{g_{c}}(\varphi^{\ast}(\phi_{i}^{\ast}g_{i}-g))=
\varphi^{\ast}\big(\nabla^{\ell}_{\varphi_{\ast}g_{c}}(\phi_{i}^{\ast}g_{i}-g)\big)
\end{align*}
and therefore, using Proposition \ref{nearmetrics},
\begin{align*}
|\nabla^{\ell}_{g_{c}}(\varphi_{i}^{\ast}g_{i}-\varphi^{\ast}g)|_{g_{c}}=|\nabla^{\ell}_{\varphi_{\ast}g_{c}}(\phi_{i}^{\ast}g_{i}-g)|_{\varphi_{\ast}g_{c}} \leq C |\nabla^{\ell}_{\varphi_{\ast}g_{c}}(\phi_{i}^{\ast}g_{i}-g)|_{g}
\end{align*}
Now, working on $\Omega\subseteq \reg$ we get 
\begin{align*}
|\nabla^{\ell}_{\varphi_{\ast}g_{c}}(\phi_{i}^{\ast}g_{i}-g)|_{g} \leq 
|\nabla^{\ell}_{\varphi_{\ast}g_{c}}(\phi_{i}^{\ast}g_{i}-g)-\nabla^{\ell}_{g}(\phi_{i}^{\ast}g_{i}-g)|_{g}+ |\nabla^{\ell}_{g}(\phi_{i}^{\ast}g_{i}-g)|_{g} \leq C\delta \sum_{j=1}^{\ell} \varepsilon/a^{j} +\delta,
\end{align*}
where we used Proposition \ref{differenzatensoremetriche} and \eqref{stimacondelta}. Finally, putting everything together and choosing $\delta$ small enough, we find
\begin{equation*}
|\nabla^{\ell}_{g_{c}}(\varphi_{i}^{\ast}-g_{c})|_{g_{c}}\leq C\delta \sum_{j=1}^{\ell} \varepsilon/a^{j} +\delta +\varepsilon/a^{\ell} \leq 3\varepsilon/a^{\ell}
\end{equation*}
and we are done.
\end{proof}
\end{claim}

Now every $\Omega_{i}\subseteq M_{i}$ satisfies the hypotheses of Theorem \ref{thm.conicalannulus} and therefore there exists a unique end $\e_{i} \in \mathcal{E}(M_{i})$ such that $V_{i}\subseteq \e_{i}(D[0,\alpha])$. Moreover $\e_{i}$ is asymptotically conical and $\nabla f_{i}\neq 0$ on $\e_{i}(D[0,\alpha])$. At this point we can construct the maps $\lambda_{i}$. First of all, note that in the above arguments $a,b,\alpha,\beta, i_{\ast\ast}$ do not depend on $\e$, and therefore Claim \ref{primo} and Claim \ref{secondo} hold for every conical end of $\X$ with the same constants $a,b,\alpha,\beta, i_{\ast\ast}$. 
Therefore we get a map $\lambda_{i}\colon \mathcal{E}_{c}(\X)\to \mathcal{E}_{c}(M_{i})$.

It is easy to see that every $\lambda_{i}$ is injective. Indeed take $\e\neq \widetilde{\e}\in \mathcal{E}_{c}(\X)$. As seen above, we have $E\cap \widetilde{E}=\emptyset$ and therefore $\Omega\cap \widetilde{\Omega}=\emptyset$. It follows that $\Omega_{i}$ and $\widetilde{\Omega}_{i}$ are disjoint for every $i$. In particular the connected components of $D_{i}(\alpha,\beta)$ they contain are disjoint and we denote them by $V_{i}$ and $\widetilde{V}_{i}$, respectively. Suppose by contradiction $\e_{i}=\widetilde{\e}_{i}$ and set $E_{i}= \e_{i}(D[0,\alpha])=\widetilde{\e}_{i}(D[0,\alpha])$. Then by Theorem \ref{thm.conicalannulus} we have $V_{i},\widetilde{V}_{i}\subseteq E_{i}$, but this is impossible, since $E_{i}$ does not contain any critical point of $f$ and hence one can apply Proposition \ref{prop.ends-shrinker} to find a contradiction.

As $\lambda_{i}$ is injective, it follows that $\#\mathcal{E}_{c}(M)\leq \#\mathcal{E}_{c}(M_{i})$ for $i$ large and 
therefore 
\begin{equation*}
\#\mathcal{E}_{c}(M)\leq \liminf_{i\to \infty} \#\mathcal{E}_{c}(M_{i}).
\end{equation*}

If $\mathcal{E}_{c}(\X)$ is not finite, one can argue exactly as in the proof of Proposition \ref{thm.limitendbound}, obtaining
\begin{equation*}
\#\mathcal{E}_{c}(\X)=\infty=\liminf_{i\to \infty}\#\mathcal{E}_{c}(M_{i}).
\end{equation*}
This concludes the proof.
\end{proof}

\section{Further applications to moduli spaces}\label{conicalends-applications-section}

As in Section \ref{applications-convergence}, Theorem \ref{thm.limitconicalbound} can be used in combination with Proposition \ref{LLWcompactness} to obtain some results about moduli spaces of Ricci shrinkers. Here, we first denote by $\m[n,\underline{\mu}]$ the moduli space of $n$-dimensional Ricci shrinkers with basepoint $p \in M$ and entropy bounded below 
$\mu\geq \underline{\mu}$. We also denote by $\m[n,\underline{\mu},E(r)]$ the subspace of Ricci shrinkers in 
$\m[n,\underline{\mu}]$ that additionally satisfy the energy condition \eqref{energycondition} if $n \geq 5$. 

Given a subspace $\m \subseteq \m[n,\underline{\mu}]$, we do not know if the number of (conical) ends is uniformly bounded, but it is nevertheless possible to consider the moduli subspace given by those Ricci shrinkers in $\m$ with \emph{exactly $N$ asymptotically conical ends}, where $N\in \N$ is arbitrary. These subspaces are denoted by $\m^{c}_{N}$. We can write
\begin{equation*}
\m=\bigsqcup_{N\in \N}\m^{c}_{N},
\end{equation*}
and therefore it is natural to study $\m^{c}_{N}$. A direct application of Theorem \ref{thm.limitconicalbound} and Proposition \ref{LLWcompactness} gives the following result.

\begin{cor} 
The boundary points of the moduli space $\m^{c}_{N}[n,\underline{\mu}]$ are given by limit singular Ricci shrinkers with at most $N$ asymptotically conical ends and the boundary points of its subspace $\m^{c}_{N}[n,\underline{\mu},E(r)]$ are given by limit orbifold Ricci shrinkers with at most $N$ asymptotically conical ends.
\end{cor}

Under additional curvature conditions, it is possible to derive some proper compactness results. To this end, if $\m$ is a moduli space, we denote by $\m^{c}_{\leq N}$ the subspace given by those Ricci shrinkers in $\m$ with \emph{at most $N$ asymptotically conical ends}. Clearly we have
\begin{equation*}
\m=\bigcup_{N\in \N}\m^{c}_{\leq N}.
\end{equation*}

We first consider the moduli space $\m[n,\underline{\mu},\underline{K}]$ of smooth Ricci shrinkers (of dimension $n$ and with entropy bounded below) with curvature uniformly bounded below $\Rm \geq \underline{K}$. We have the following result.

\begin{cor} 
The moduli space $\m^c_{\leq N}[n,\underline{\mu},\underline{K}]$ is compact with respect to the smooth pointed Cheeger-Gromov convergence.
\end{cor}

\begin{proof}
Given a sequence $(M_{i},g_{i},f_{i},p_{i})_{i \in \N}$ in $\m^c_{\leq N}[n,\underline{\mu},\underline{K}]$, thanks to Proposition \ref{LLWcompactness} we have convergence, up to a subsequence, to a smooth Ricci shrinker $(M,g,f,p)$. 
Since the convergence is smooth, $M$ has entropy and curvature bounded below by $\underline{\mu}$ and $\underline{K}$, respectively. We can then use Theorem \ref{thm.limitconicalbound} to conclude that $\#\E_{c}(M)\leq N$, and therefore the limit lies again in 
$\m^c_{\leq N}[n,\underline{\mu},\underline{K}]$.
\end{proof}

Finally, if $n=4$, we use again $\m[4,\underline{\mu}, A,B]$ to denote the moduli space of $4$-dimensional Ricci shrinkers with entropy lower bound, bounded scalar curvature $S\leq A$ and $|\Rm|\leq B$ on $D[0,r_{1})$, where $r_{1}>0$ is as in Theorem \ref{mw-primo} (in particular depending on $A$), as already considered in Section \ref{applications-convergence}. Remember that Theorem \ref{mw-primo} yields a uniform global bound $|\Rm|\leq C$ on $M$ for every element of this moduli space, where $C=C(A,B)>0$. We therefore find the following corollary.
 
\begin{cor} 
The moduli space $\m^c_{\leq N}[4,\underline{\mu}, A,B]$ is compact with respect to the smooth pointed Cheeger-Gromov convergence.
\end{cor}

\appendix
\section{Proofs of the Propositions from Section \ref{sectionends}}\label{appendixends}

\begin{proof}[Proof of Proposition \ref{prop.ends-general}]\
\begin{enumerate}
\item By definition, $\e(K)$ is a connected component of $X\setminus K$. Suppose towards a contradiction that $\overline{\e(K)}$ is compact. Then $H=K\cup \overline{\e(K)}$ is compact and $\e(H)\subseteq \e(K)\subseteq H$, 
which is impossible since $\e(H)\subseteq X\setminus H$.

\item We start by proving the following general fact.
\smallskip

\begin{claim}\label{relcptcc} Let $K\in \mathcal{K}_{X}$ and consider a collection $\mathcal{A}$ of relatively compact connected components of $X\setminus K$. Then $\bigcup_{A\in \mathcal{A}}A$ is relatively compact. 
\end{claim} 
\begin{proof}
Let $U\subseteq X$ be a relatively compact open subset containing $K$. Then $U$ intersects every $A\in \mathcal{A}$. Indeed suppose by contradiction that $A\cap U=\emptyset$ for some $A\in \mathcal{A}$. From $A\subseteq X\setminus U$ follows $\overline{A}\subseteq X\setminus U$ and therefore $\overline{A}=A$, since $\overline{A}$ is connected and contained in $X\setminus K$. Hence $A$ is open and closed in $X$, in contradiction to the connectedness of $X$. Now if $A\in \mathcal{A}$ then $A\subseteq U$ or $A\cap \partial U\neq \emptyset$. Indeed, if $A\nsubseteq U$ but $A\cap \partial U=\emptyset$ we could write $A=(A\cap U)\cup (A\cap (X\setminus \overline{U}))$, that is as the union of two disjoint open subset, which is impossible being $A$ connected. At this point we have $\bigcup_{A\in \mathcal{A}}A\subseteq U\cup \bigcup_{A\in \mathcal{A}_{\ast}}A$, where $\mathcal{A}_{\ast}=\{A\in \mathcal{A}\colon A\cap \partial U\neq \emptyset\}$. Observe that since $\{A\cap \partial U\colon A\in \mathcal{A}_{\ast}\}$ is a cover of the compact set $\partial U$ formed by pairwise disjoint open subsets it must be finite. In other words $\mathcal{A}_{\ast}$ is finite and we conclude, being $U\cup \bigcup_{A\in \mathcal{A}_{\ast}}A$ relatively compact.
\end{proof}

Returning to the main statement, suppose towards a contradiction that there is no end relative to $K$, that is, every connected component of $X\setminus K$ is relatively compact. It follows by Claim \ref{relcptcc} that $X\setminus K$ is relatively compact too, and therefore $X=K\cup (X\setminus K)$ is compact.

\item Let $E$ be an end relative to $K$. Since $E$ is not relatively compact $E\setminus H\neq \emptyset$ and
the set $\mathcal{A}=\{A\in \pi_{0}(X\setminus H)\colon A\subseteq E\}$ is not empty. Moreover $E\subseteq H\cup \left(\cup_{A\in \mathcal{A}}A\right)$. If all the elements in $\mathcal{A}$ were relatively compact then $E$ would be relatively compact too by Claim \ref{relcptcc}, a contradiction. Therefore there exists some element in $\mathcal{A}$ with non-compact closure, thus $E$ contains an end relative to $H$.
 
\item We start proving injectivity. Let $\e,\f \in \E(X)$ be such that $\e(K_{i})=\f(K_{i})$ for every $i\in \N$, and suppose by contradiction that $\e\neq \f$. Then there exists $K$ compact such that $\e(K)\cap \f(K)=\emptyset$. For $i$ large, we have $K\subseteq K_{i}$ and therefore $(\e(K_{i})\cap \f(K_{i}))\subseteq (\e(K)\cap \f(K))=\emptyset$, which is impossible. To show surjectivity, consider $(U_{i})_{i \in \N}\in \Pi_{i \in \N} \pi_{0}(X\setminus K_{i})$ such that $U_{i+1}\subseteq U_{i}$ for every $i\in \N$. We define a function $\e \colon \mathcal{K}_{X}\to \mathcal{P}(X)$ as follows. If $K=K_i$ set $\e(K)=U_{i}$. Given a general $K\in \mathcal{K}_{X}$ fix an index $j$ such that $K\subseteq K_{j}$ and define $\e(K)$ as the unique connected component of $X\setminus K$ that contains $\e(K_{j})$. This function $\e$ is well defined (i.e. $\e(K)$ does not depend on the choice of $K_{j}$) and is an end. Indeed $\e(H)\subseteq \e(K)$ if $K\subseteq H$ and arguing as above one obtains that every $\e(K)$ is not relatively compact.
 
\item Consider an exhaustion $\{K_{i}\}_{i\in \N}$ of $X$ with $K_{0}=K$. We define $\e\in \mathcal{E}(X)$ recursively as follows: $\e(K_{0})=E$ while
$\e(K_{i+1})$ is an end relative to $K_{i+1}$ that is contained in $\e(K_{i})$; here we used that $\E(K_{i+1},X)\to \E(K_{i},X)$ is onto.
This proves the existence statement. If $E$ satisfies the additional hypothesis it is straightforward to conclude that the end $\e$ constructed above is unique.\qedhere
\end{enumerate}
 \end{proof}

\begin{proof}[Proof of Proposition \ref{numerofinitofini}]
By Point 3 of Proposition \ref{prop.ends-general}, if $K\subseteq H$, the natural map $ \mathcal{E}(H,X)\to \mathcal{E}(H,X)$ is surjective, and therefore $n(K)\leq n(H)$. If $n(\cdot)$ is bounded then it has a maximum being integer valued. It only remains to prove the final identity. Denote 
\begin{equation*}
n^{\ast}:=\sup_{K\in \mathcal{K}_{X}}n(K).
\end{equation*}
For $K\in \mathcal{K}_{X}$ consider the map $\mathcal{E}(X)\to \mathcal{E}(K,X)$ given by $\e\mapsto \e(K)$. This map is well defined and onto thanks to Point 5 in Proposition \ref{prop.ends-general} and therefore $n(K)\leq \#\mathcal{E}(X)$. Taking the supremum over $K$ we deduce $n^{\ast}\leq \#\mathcal{E}(X)$. In particular $\#\mathcal{E}(X)=\infty$ if $n^{\ast}=\infty$. Now suppose instead that $n^{\ast}$ is finite. Then there exists $K$ compact such that $n(K)=n^{\ast}$. Moreover $n(H)=n^{\ast}$ for every $H\in \mathcal{K}_{X}$ containing $K$. Thanks to Points 4 and 5 in Proposition \ref{prop.ends-general}, we can see that the map $\mathcal{E}(H,X)\to \mathcal{E}(K,X)$ is a bijection and therefore $\mathcal{E}(X)\to \mathcal{E}(K,X)$ is a bijection as well. This gives $\#\mathcal{E}(X)=n^{\ast}$ and we are done.
\end{proof}

\begin{proof}[Proof of Proposition \ref{prop.ends-shrinker}]
Let $N=\reg \setminus \Crit$, where $\Crit$ is the set of critical points of $f$. Then $N$ is open in $\reg$ and contains $W$. Take $x_{0}\in N$ and let $\gamma \colon J\to N$ be the maximal integral curve of $\nabla f/|\nabla f|^{2}$ with $\gamma(0)=x_{0}$. Differentiating $f(\gamma(t))$ we get $\frac{d}{dt}f(\gamma(t))=1$ and therefore 
\begin{equation}\label{evoluzionef}
 f(\gamma(t))=f(x_{0})+t, \quad \forall t\in J.  
\end{equation} 
Using the definition of $\varrho$ we deduce
\begin{equation}\label{evoluzionerho}
\varrho^{2}(\gamma(t))=\varrho^{2}(x_{0})+4t, \quad \forall t\in J.
\end{equation}
Suppose $x_{0}\in W$ and denote $\xi=\varrho(x_{0})$. We claim that $I\subseteq \varrho(\gamma(J))$. Indeed let $J=(\alpha,\beta)$ and take $\xi \leq \zeta \in I$, then there exists $t_{\ast} \in [0,\beta)$ such that $\varrho(\gamma(t_{\ast}))>\zeta$. Suppose by contradiction that this is not the case. Since $\varrho$ is increasing along $\gamma$ we have $\gamma(t)\in D[\xi,\zeta]\cap W$ for every $t\in [0,\beta)$ and in particular $\beta<\infty$. Now $D[\xi,\zeta]\cap W\subseteq N$ and it is compact, being the intersection of 
the compact subset $D[\xi,\zeta]\subseteq D(I)$ and $W$, which is closed in $D(I)$. This contradicts the escape lemma (see e.g. \cite[Lemma 9.19]{lee}), and we conclude. In the same way one sees that for every $\xi \geq \zeta \in I$ there exists $t_{\ast}\in (\alpha,0]$ such that $\varrho(\gamma(t_{\ast}))<\zeta$. This proves $\varrho(W)=I$. We can use this to show the other four points of the proposition.

\begin{enumerate}
\item If $I$ is an open interval, all the claims are obvious. Suppose $I$ is not open. Since by assumption $W$ does not contains any critical point of $\varrho$ it is an $n$-dimensional submanifold with boundary of $\reg$. From $\varrho(W)=I$ we obtain $\partial W=D(\partial I \cap I)\cap W \neq \emptyset$ and $W^{\circ}=D(I^{\circ}) \cap W$. Since $W$ is connected also $W^{\circ}$ is connected and therefore it coincides with the unique connected component of $D(I^{\circ})$ contained in $W$.

\item If $\xi \in I$ then $W(\xi)=D(\xi)\cap W\neq \emptyset$. If follows from the preceding point and \eqref{evoluzionerho} that for every $x_{0}\in W(\xi)$ the maximal integral curve of $\nabla f/|\nabla f|^{2}$ starting from $x_{0}$ is defined at least on the interval $I(\xi)=\{t\in \R\colon t=(s^{2}-\xi^{2})/4, s \in I\}$ and is contained in $W$ for such times. Therefore we can consider the map $\omega\colon W(\xi)\times I(\xi)\to W$ induced by the local flow of $\nabla f/|\nabla f|^{2}$. Then $\omega$ is smooth and injective, as one can see using \eqref{evoluzionerho}. Now take $y_{0}\in W$. If $y_{0}\in W(\xi)$ then $y_{0}=\omega(y_{0},0)$, otherwise let $\eta\colon J\to N$ be the integral curve of $X$ starting from $y_{0}$. The same argument used above shows that there exists $t_{\ast}\in J$ such that $z_{0}=\eta(t_{\ast})\in W(\xi)$. In particular $t_{\ast}=(\varrho^{2}(y_{0})-\xi^{2})/4$ and $-t_{\ast}\in I$. Then $y_{0}=\omega(z_{0},-t_{\ast})$ and therefore $\omega$ is a diffeomorphism. As $W$ is connected, $W(\xi)$ is connected too and is the unique connected component of $D(\xi)$ that is contained in $W$. The case of $\nabla f$ can be proven in the same way and we omit the proof.

\item Let $\omega \colon W(\xi) \times I(\xi) \to W$ be as in Point 2. Then $\varrho^{2}(\omega(x,t))=\xi^{2}+4t$ and $D(J)\cap W=\omega(W(\xi)\times J(\xi))$, which is connected since $J(\xi)$, defined analogously as $I(\xi)$, is an interval.

\item Since $\varrho(W^{\circ})=I^{\circ}$ and $\sup I=\infty$, it follows by \eqref{rhogrowth-singular} that $d(x,p)$ is unbounded on $W^{\circ}$. Denoting $a=\inf I$ we deduce that $W^{\circ}$ is an end of $\X$ relative to $D[0,a]$, as claimed. Now consider the exhaustion of $\X$ by the compact subsets given by $K_{i}=D[0,a+i]$ for $i \in \N$. For every $i$ we have $\X\setminus K_{i}=D(a+i,\infty)$ and thanks to the preceding point there exists a unique connected component of this set that is contained in $W$. As above, denote this component by $W(a+i,\infty)$. Being unbounded, $W(a+i,\infty)$ is and end of $\X$ relative to $D[0,a+i]$ and we conclude using Point 4 of Proposition \ref{prop.ends-general}.\qedhere
\end{enumerate}
\end{proof}

\section{Nearby Riemannian metrics}\label{appendix-nearbymetrics}

In the following, $C$ will always denote a positive constant (depending on $n$) that possibly changes from line to line. The three propositions of this appendix are all obtained by direct calculations, but we include brief sketches of proofs for completeness.

\begin{prop}\label{nearmetrics} 
Let $g,h$ be Riemannian metrics on a smooth manifold $M$ and $\varepsilon\colon M\to (0,1/2)$ a function such that $|g-h|_{g}<\varepsilon$. Then for every vector field $v$ on $M$, we have
\begin{equation}\label{Lipschitzequiv}
(1-\varepsilon)|v|^{2}_{g}\leq |v|^{2}_{h}\leq (1+\varepsilon)|v|^{2}_{g}.
\end{equation}
In particular $g$ and $h$ are Lipschitz equivalent. Moreover $|g-h|_{h}\leq C\varepsilon$ and, more generally, 
\begin{equation*}
|T|_{h}\leq C |T|_{g}
\end{equation*}
for every tensor field $T$, where $C>0$ is a constant depending only on $n$ and the type of $T$.
\end{prop}

\begin{proof}
Note that (even without the assumption $\varepsilon < 1/2$)
\begin{equation*}
|(g-h)(v,v)|\leq |g-h|_{g}|v|^{2}_{g}\leq \varepsilon |v|^{2}_{g}
\end{equation*}
for every vector field $v$. Writing $h(v,v)=(h-g)(v,v)+g(v,v)$ both bounds in \eqref{Lipschitzequiv} follow. If $\varepsilon < 1/2$ this obviously gives the Lipschitz equivalence of the metrics.

Now let $p\in M$ and choose normal coordinates for $g$ at $p$. Then $g_{ij}=\delta_{ij}$ and
\begin{equation*}
|g-h|_{g}^{2}=\sum_{i,j=1}^{n}|\delta_{ij}-h_{ij}|^{2}\leq \varepsilon^{2}.
\end{equation*}
This means that the Frobenius norm of the matrix $I-h$ is less than one and therefore
\begin{equation*}
\sum_{i,j=1}^{n}|\delta^{ij}-h^{ij}|^{2}\leq \Big(\frac{\varepsilon}{1-\varepsilon}\Big)^{2} \leq 4\varepsilon^2.
\end{equation*}
At this point, we conclude that
\begin{equation*}
|g-h|_{h}^{2}=\sum_{a,b,i,j=1}^{n}h^{ai}h^{bj}(\delta_{ab}-h_{ab})(\delta_{ij}-h_{ij})\leq n^{4}\Big(1+\frac{\varepsilon}{1-\varepsilon}\Big)^{2}\varepsilon^{2} =C\varepsilon^{2}.
\end{equation*}
A similar argument gives the result for general tensor fields $T\in \Gamma(T^{\ell}_{k}M)$ with $C=C(n,\ell,k)$.
\end{proof}

\begin{prop}\label{curvaturametrichevicine} 
Let $g,h$ be two Riemannian metrics on $M$ and suppose 
\begin{equation*}
|\nabla^{\ell}_{g}(g-h)|_{g}\leq \varepsilon_{\ell}
\end{equation*} 
for $0\leq \ell \leq 2$ and functions $\varepsilon_{\ell}\colon M\to (0,1/2)$. Then
\begin{equation*}
|\Rm(g)-\Rm(h)|_{g}\leq C(\varepsilon_{1}^{2}+\varepsilon_{2})
\end{equation*}
on $M$, with a constant $C>0$ depending only on $n$.
\end{prop}

\begin{proof}
Choose normal coordinates for $g$ at a point $p\in M$. All calculations below are carried out at the origin of these normal coordinates. Write $A=\Rm(g)-\Rm(h)$, and denote by $\Gamma$ and $\gamma$ the Christoffel symbols of $g$ and $h$, respectively.  We have
\begin{equation}\label{eqAijkl}
A_{ijk}^{\ell}= \tfrac{\partial}{\partial x^{j}}\big(\gamma_{ik}^{\ell}-\Gamma_{ik}^{\ell}\big) -
\tfrac{\partial}{\partial x^{i}}\big(\gamma_{jk}^{\ell}-\Gamma_{jk}^{\ell}\big) + \big(\Gamma_{jk}^{s}\Gamma_{is}^{\ell}-\Gamma_{ik}^{s}\Gamma_{js}^{\ell}\big) -\big(\gamma_{jk}^{s}\gamma_{is}^{\ell}-\gamma_{ik}^{s}\gamma_{js}^{\ell}\big).
\end{equation}
Recall that the Christoffel symbols of $g$ vanish at the origin of normal coordinates and therefore the third term above vanishes. Moreover, we obtain
\begin{align}
\nabla_b h_{cd} &= \frac{\partial h_{cd}}{\partial x^b} - \Gamma_{bc}^s h_{sd} - \Gamma_{bd}^s h_{cs} = \frac{\partial h_{cd}}{\partial x^b}, \label{eqfirstderchange} \\
\nabla^2_{ab}h_{cd} &=\nabla_a \nabla_b h_{cd} = \frac{\partial^2 h_{cd}}{\partial x^a \partial x^b} - \frac{\partial \Gamma_{bc}^s}{\partial x^a} h_{sd} -  \frac{\partial\Gamma_{bd}^s}{\partial x^a} h_{cs}. \label{eqsecondderchange}
\end{align}
Thus, we find for the first term on the right hand side of \eqref{eqAijkl},
\begin{align*}
\frac{\partial}{\partial x^{j}}\Big(\gamma_{ik}^{\ell}-\Gamma_{ik}^{\ell}\Big) &= \frac{1}{2} \frac{\partial h^{m\ell}}{\partial x^j} 
\bigg(\frac{\partial h_{km}}{\partial x^i} + \frac{\partial h_{im}}{\partial x^k} - \frac{\partial h_{ik}}{\partial x^m}\bigg)
+ \frac{1}{2} h^{m\ell}\bigg(\frac{\partial^2 h_{km}}{\partial x^j \partial x^i} + \frac{\partial^2 h_{im}}{\partial x^j \partial x^k} - \frac{\partial^2 h_{ik}}{\partial x^j \partial x^m}\bigg) - \frac{\partial \Gamma_{ik}^{\ell}}{\partial x^j} \\
&= \frac{1}{2} \frac{\partial h^{m\ell}}{\partial x^j} 
\Big(\nabla_i h_{km} + \nabla_k h_{im}- \nabla_m  h_{ik}\Big) + \frac{1}{2} h^{m\ell}\Big(\nabla^2_{ji} h_{km} + \nabla^2_{jk} h_{im} - \nabla^2_{jm} h_{ik}\Big).
\end{align*}
Here we changed the first derivatives of $h$ using \eqref{eqfirstderchange} and the second derivatives using \eqref{eqsecondderchange}, noting that four of the six Christoffel symbols that arise from this latter change cancel each other while the remaining two cancel with $- \frac{\partial}{\partial x^j}  \Gamma_{ik}^{\ell}$.

As seen in Proposition \ref{nearmetrics}, from the hypothesis $|h-g|_g\leq \varepsilon_0$, we obtain $|h_{ij}-\delta_{ij}|\leq \varepsilon_{0}$ and $|h^{ij}-\delta^{ij}|\leq C\varepsilon_{0}$. In particular $|h_{ij}|\leq 1+\varepsilon_{0}$ and $|h^{ij}|\leq 1+C\varepsilon_{0}$. Similarly, from $|\nabla h|_g \leq \varepsilon_1$, we get $|\nabla_{i}h_{jk}|\leq \varepsilon_{1}$ and with \eqref{eqfirstderchange} we conclude $|\partial_{i}h_{jk}|\leq \varepsilon_{1}$ and as a simple consequence $|\partial _{i}h^{jk}|\leq C\varepsilon_{1}$. At this point, combining the above with the assumption $|\nabla^2 h|_g\leq \varepsilon_2$, we conclude
\begin{equation*}
\Big| \tfrac{\partial}{\partial x^{j}}\big(\gamma_{ik}^{\ell}-\Gamma_{ik}^{\ell}\big) \Big| \leq C\varepsilon_1^2 + C(1+C\varepsilon_0) \varepsilon_2 \leq C(\varepsilon_1^2+\varepsilon_2).
\end{equation*}
Clearly, the second term on the right hand side of \eqref{eqAijkl} can be estimated in the same way, while for the third term one can use
\begin{equation*}
|\gamma_{ij}^k| \leq C(1+C\varepsilon_0)\varepsilon_1 \leq C\varepsilon_1.
\end{equation*}
This immediately yields $|A_{ijk}^{\ell}|\leq C(\varepsilon_1^2+\varepsilon_2)$, and the claim follows.
\end{proof}

\begin{prop}\label{differenzatensoremetriche} 
Let $g,h$ be two Riemannian metrics on the smooth manifold $M$ and suppose 
\begin{equation*}
|\nabla^{\ell}_{g}(g-h)|_{g}<\varepsilon_{\ell}
\end{equation*} 
for $0\leq \ell \leq 2$ and functions $\varepsilon_{\ell}\colon M\to (0,1/2)$. 
If $T$ is a symmetric $(0,2)$-tensor field then 
\begin{equation*}
|\nabla_{g}T-\nabla_{h}T|_{g}\leq C\varepsilon_{1}|T|_{g}
\end{equation*}
and
\begin{equation*}
|\nabla^{2}_{g}T-\nabla^{2}_{h}T|_{g}\leq C(\varepsilon_{1}^{2}+\varepsilon_{2})|T|_{g}+C\varepsilon_{1}|\nabla_{g}T|_{g},
\end{equation*}
where $C>0$ is a constant depending only on $n$.
\end{prop}

\begin{proof}
Denote by $\nabla, D$ and $\Gamma,\gamma$ the Levi-Civita connections and the Christoffel symbols of $g$ and $h$, respectively. Choose normal coordinates for $g$ at a point $p\in M$. Working in the origin of these normal coordinates, we have seen in the preceding proofs that
\begin{equation*}
|h^{ij}|\leq C,\quad |\nabla_{i}h_{jk}|\leq \varepsilon_{1},\quad |\partial_{i}h^{jk}|\leq C\varepsilon_{1},\quad |\gamma_{ij}^{k}|\leq C\varepsilon_{1},
\quad |\nabla^{2}_{ij}h_{km}|\leq \varepsilon_{2}.
\end{equation*}
Computing the first covariant derivatives of $T$, we easily obtain
\begin{equation*}
\nabla_{i}T_{jk} -D_{i}T_{jk}=\gamma_{ij}^{\ell}T_{\ell k} + \gamma_{ik}^{\ell} T_{j\ell} 
\end{equation*} 
and therefore
\begin{equation*}
|\nabla T-  DT|_{g}^{2}= \sum_{i,j,k=1}^n \! \bigg(\sum_{\ell=1}^n \gamma_{ij}^{\ell}T_{\ell k} + \gamma_{ik}^{\ell}T_{j\ell}\bigg)^{2}
\leq C^{2}\varepsilon_{1}^{2} \sum_{i,j,k=1}^n \! \bigg(\sum_{\ell=1}^n |T_{\ell k}|+|T_{j\ell}|\bigg)^{2} \leq C^{2}\varepsilon_{1}^{2}|T|_{g}^{2},
\end{equation*}
proving the first claim.

The difference between the second covariant derivatives can be written as
\begin{equation*}
\nabla^{2}_{ij}T_{kl}-D^{2}_{ij}T_{kl}= \frac{\partial}{\partial x^{i}}\big(\nabla_{j}T_{k\ell}-D_{j}T_{k\ell}\big)+ \gamma_{ij}^{s}D_{s}T_{k\ell}+\gamma_{ik}^{s}D_{j}T_{s\ell}+\gamma_{i\ell}^sD_{j}T_{ks}.
\end{equation*}
The last three terms are estimated by the bound on $\gamma_{ij}^k$ and by using the above to change $DT$ to $\nabla T$. The first term on the right hand side can instead be rewritten as
\begin{equation*}
\frac{\partial}{\partial x^{i}}\big(\nabla_{j}T_{k\ell}-D_{j}T_{k\ell}\big) =
\frac{\partial}{\partial x^{i}}\big(\gamma_{jk}^{s}-\Gamma_{jk}^{s}\big) T_{s\ell} +
\frac{\partial}{\partial x^{i}}\big(\gamma_{j\ell}^{s}-\Gamma_{j\ell}^{s}\big) T_{ks}
 +\gamma_{jk}^{s}\frac{\partial T_{s\ell}}{\partial x^{i}} + \gamma_{j\ell}^{s}\frac{\partial T_{ks}}{\partial x^{i}}.
\end{equation*}
The proof then follows by the estimates for $\gamma_{ij}^k$ as well as for the derivatives of the difference of the Christoffel symbols $\gamma$ and $\Gamma$, exactly as in the proof of Proposition \ref{curvaturametrichevicine} above.
\end{proof}


\makeatletter
\def\@listi{%
  \itemsep=0pt
  \parsep=1pt
  \topsep=1pt}
\makeatother
{\fontsize{10}{12}\selectfont

}
\vspace{-2mm}
\printaddress

\end{document}